\documentclass[a4paper,12pt]{amsart}
\usepackage{fullpage}

\usepackage{amsmath, amssymb, amsthm,enumitem,bm,pict2e,xcolor,tikz-cd,indentfirst,xparse,ytableau,diagbox}
\usepackage[foot]{amsaddr}
\usepackage[backref=page]{hyperref}

\hypersetup{
    colorlinks=true,
    linkcolor=blue,
    filecolor=magenta,      
    urlcolor=cyan,
}

\theoremstyle{definition}
\newtheorem{theorem}{Theorem}
\newtheorem*{theorem*}{Theorem}
\numberwithin{theorem}{section}
\newtheorem{definition}[theorem]{Definition}
\newtheorem*{definition*}{Definition}
\newtheorem{proposition}[theorem]{Proposition}
\newtheorem{lemma}[theorem]{Lemma}
\newtheorem{remark}[theorem]{Remark}
\newtheorem{example}[theorem]{Example}
\newtheorem{cor}[theorem]{Corollary}

\newtheorem{innercustomthm}{Theorem}
\newenvironment{customthm}[1]
  {\renewcommand\theinnercustomthm{#1}\innercustomthm}
  {\endinnercustomthm}

\DeclareMathOperator{\grad}{grad}

\DeclareMathOperator{\id}{id}

\DeclareMathOperator{\gr}{gr}

\DeclareMathOperator{\initial}{in}

\DeclareMathOperator{\MultiProj}{MultiProj}
\DeclareMathOperator{\Proj}{Proj}

\newcommand{\la}{\lambda}
\newcommand{\om}{\omega}
\newcommand{\cU}{\mathcal U}
\newcommand{\cL}{\mathcal L}
\newcommand{\cJ}{\mathcal J}
\newcommand{\cO}{\mathcal O}

\newcommand{\fg}{\mathfrak{g}}
\newcommand{\fsl}{\mathfrak{sl}}

\newcommand{\fh}{\mathfrak{h}}

\newcommand{\bC}{\mathbb{C}}
\newcommand{\bP}{\mathbb{P}}
\newcommand{\bR}{\mathbb{R}}
\newcommand{\bZ}{\mathbb{Z}}
\newcommand{\bd}{{\mathbf d}}

\newcommand{\al}{\langle}
\newcommand{\ar}{\rangle}
\renewcommand{\mod}{\ \mathrm{mod}\ }

\DeclareFontFamily{U}{mathx}{\hyphenchar\font45}
\DeclareFontShape{U}{mathx}{m}{n}{
      <5> <6> <7> <8> <9> <10>
      <10.95> <12> <14.4> <17.28> <20.74> <24.88>
      mathx10
      }{}
\DeclareSymbolFont{mathx}{U}{mathx}{m}{n}
\DeclareMathSymbol{\bigtimes}{1}{mathx}{"91}

\mathcode`l="8000
\begingroup
\makeatletter
\lccode`\~=`\l
\DeclareMathSymbol{\lsb@l}{\mathalpha}{letters}{`l}
\lowercase{\gdef~{\ifnum\the\mathgroup=\m@ne \ell \else \lsb@l \fi}}%
\endgroup

\mathchardef\newbracket=\mathcode`)
\mathcode`)="8000
\begingroup
\catcode`) \active
\gdef){\nolinebreak\newbracket}
\endgroup

\mathchardef\newcomma=\mathcode`,
\mathcode`,="8000
\begingroup
\catcode`, \active
\gdef,{\nolinebreak\newcomma}
\endgroup

\title[Chain-order polytopes: degenerations, tableaux and bases]{Chain-order polytopes: toric degenerations, Young tableaux and monomial bases}

\author{Igor Makhlin}
\address{Technical University of Berlin}
\email{iymakhlin@gmail.com}

\begin{document}

\maketitle

\begin{abstract} 
Our first result realizes the toric variety of every marked chain-order polytope (MCOP) of the Gelfand--Tsetlin poset as an explicit Gr\"obner (sagbi) degeneration of the flag variety. This generalizes the Sturmfels/Gonciulea--Lakshmibai/Kogan--Miller construction for the Gelfand--Tsetlin degeneration to the MCOP setting. The key idea of our approach is to use pipe dreams to define realizations of toric varieties in Pl\"ucker coordinates. We then use this approach to generalize two more well-known constructions to arbitrary MCOPs: standard monomial theories such as those given by semistandard Young tableaux and PBW-monomial bases in irreducible representations such as the FFLV bases. In an addendum we introduce the notion of semi-infinite pipe dreams and use it to obtain an infinite family of poset polytopes each providing a toric degeneration of the semi-infinite Grassmannian.
\end{abstract}


\section*{Introduction}

In the study of flat degenerations of various geometric objects particular attention is paid to degenerations of Lie-theoretic varieties such as flag and Schubert varieties. Recent decades have seen a broad range of results in this field (\cite{GL,Ch,Ca,KM,AB,Fe,FFL2,Ka,GHKK,FaFL1} and many others) shape into a theory at the interplay of algebraic geometry, combinatorics, representation theory and commutative algebra.

In particular, a constantly-growing array of works provides new methods of obtaining such degenerations. These papers often proceed by attaching a degeneration to every combinatorial or algebraic object of a certain form. Examples include adapted decompositions in the Weyl group (\cite{Ca}), specific valuations on the function field (\cite{Ka}), specific birational sequences (\cite{FaFL2}) and coherent matching fields (\cite{CM}). While these notions and the discovered connections are of great interest, not many explicit yet general construction methods are known for these attached objects. This leads to a certain shortage of concrete recipes for constructing flat degenerations that would work in a general situation.

Let us consider the fundamental case of toric degenerations of type A flag varieties which will be discussed in this paper. The first and best-known such degeneration is the Gelfand--Tsetlin degeneration due to~\cite{Stu,GL,KM} given by the Gelfand--Tsetlin polytope of~\cite{GT}. A more recent construction is the FFLV degeneration due to~\cite{FFL2,FFFM} given by the Feigin--Fourier--Littelmann--Vinberg polytope of~\cite{FFL1}. Until recently these two constructions remained the only known explicit definitions of polytopes that would provide a toric degeneration for every flag variety of type A (despite it being well known that plenty of other degenerations exist, some numerical results are found in~\cite{BLMM}). A step towards closing this gap was made by the paper~\cite{Fu} which proved that each of $2^{n(n-1)/2}$ explicitly defined polytopes provides a toric degeneration for a flag variety of type $\mathrm A_{n-1}$. Let us give some combinatorial background for this result.

The notion of poset polytopes originates in~\cite{St} where two polytopes were associated with every poset: the order and the chain polytope. These definitions have been substantially generalized, in particular, Gelfand-Tsetlin and FFLV polytopes are, respectively, marked order and marked chain polytopes of the Gelfand--Tsetlin poset, see~\cite{ABS}. The even broader class of marked chain-order polytopes (MCOPs) has been defined and studied in~\cite{FF,FFLP,FFP}. In the case of the Gelfand--Tsetlin poset every such polytope $\cO_{O,C}(\la)$ is given by a decreasing $n$-tuple of integers $\la$ and a partition $O\sqcup C$ of the poset (enumerated by positive roots) into two subsets, hence $2^{n(n-1)/2}$. The case $C=\varnothing$ then provides the Gelfand--Tsetlin polytope while the case $O=\varnothing$ the FFLV polytope. In~\cite{Fu} it is proved that the toric variety of every such polytope is a flat degeneration of the flag variety corresponding to $\la$.

Now, it must be noted that the main object of study in~\cite{Fu} are Newton--Okounkov bodies. Toric degenerations are obtained as a consequence via the results in~\cite{A} which, in a general setting, show that Newton--Okounkov polytopes provide toric degenerations. The initial motivation of this paper is to give a more direct approach to the realization of these flat degenerations. Specifically, we aim to generalize the explicit Gr\"obner degeneration constructions known for the Gelfand--Tsetlin and FFLV cases to all MCOPs. We now sketch what such a generalization would look like.


The flag variety is realized by the Pl\"ucker ideal $I$ in the polynomial ring in the Pl\"ucker variables $X_{i_1,\dots,i_k}$. Meanwhile, the toric variety of $\cO_{O,C}(\la)$ is realized by a toric ideal $I^{O,C}$ in the polynomial ring in variables $X_J$ labeled by order ideals $J$ in the Gelfand--Tsetlin poset. To realize the toric variety as a Gr\"obner degeneration of the flag variety we are to define an isomorphism between the two polynomial rings which would map $I^{O,C}$ to an initial ideal $\initial_< I$ with respect to some monomial order $<$. This is precisely the approach used in~\cite{GL} and~\cite{FFFM} and the challenge is to find a common generalization of the two constructions that would apply to all MCOPs. The solution is given by the notion of pipe dreams which can be viewed as combinatorial rule for associating a permutation $w_M$ with every subset $M$ of the positive roots (or elements of the Gelfand--Tsetlin poset). This notion originates in~\cite{FK} and plays an instrumental role in the algebraic combinatorics of flag varieties (\cite{BB,K,KnM,KM,KST} and others). We prove the following.
\begin{customthm}{A}[cf.\ Theorem~\ref{degenmain}]
For every partition $O\sqcup C$ of the Gelfand--Tsetlin poset and every order ideal $J$ we have a subset $M_J$ in the poset and an integer $k_J$ such that the following holds. The map \[\psi:X_J\mapsto X_{w_{M_J}(1),\dots,w_{M_J}(k_J)}\] is an isomorphism and there exists a monomial order $<$ for which $\psi(I^{O,C})=\initial_< I$. This realizes the toric variety of $\cO_{O,C}(\la)$ as a degeneration of the flag variety.
\end{customthm}

Both the Gelfand--Tsetlin and the FFLV degenerations are accompanied by a collection of intimately related algebraic and geometric constructions. Two well-known examples include standard monomial theories and PBW-monomial bases in irreducible representations. Our subsequent results generalize these two concepts to the MCOP setting.

The classical example of a standard monomial theory is the basis in the Pl\"ucker algebra given by monomials $\prod_j X_{i^j_1,\dots,i^j_{k_j}}$ for which the Young tableau with elements $i^a_b$ is semistandard. 
Each tableau defines a point $x$ with coordinate $x_{r,s}$ equal to the number of elements $s$ in row $r$. The convex hull of such points over all semistandard tableaux of shape $\la$ is unimodularly equivalent to the Gelfand--Tsetlin polytope, this observation is essentially due to~\cite{KM}.
The FFLV counterpart was defined in~\cite{Fe} in the form of PBW-semistandard Young tableaux which provide a basis in the Pl\"ucker algebra in a similar way and have a similar connection with FFLV polytopes. Our generalization has the following form.
\begin{customthm}{B}[cf. Theorem~\ref{standardmain}, Proposition~\ref{tableauxpolytope}]
For every partition $O\sqcup C$ there exists a notion of $(O,C)$-semistandard Young tableaux such that monomials $\prod_j X_{i^j_1,\dots,i^j_{k_j}}$ for which the tableau with elements $i^a_b$ is $(O,C)$-semistandard (standard monomials) provide a basis in the Pl\"ucker algebra. 
The convex hull of points $x$ defined as above over all $(O,C)$-semistandard Young tableaux of shape $\la$ is unimodularly equivalent to $\cO_{O,C}(\la)$.
\end{customthm}

For an integral dominant weight $\la$ consider the irreducible representation $V_\la$ with highest weight vector $v_\la$. Let $P$ denote the set of pairs $1\le i\le j\le n$. Every pair in $P$ with $i<j$ corresponds to a negative root vector $f_{i,j}$. Every point $x\in\bZ_{\ge0}^P$ corresponds to the product $f^x=\prod_{i<j} f_{i,j}^{x_{i,j}}$ (where we fix an ordering of the roots). If a collection $A$ of $\dim V_\la$ points is particularly nice, the set $\{f^x v_\la\}_{x\in A}$ is a basis in $V_\la$ known as a ``monomial basis''. For general $\la$ there are two known constructions of polytopes in $\bR^P$ whose integer point sets are particularly nice in this sense. The first is the FFLV polytope which provides the FFLV basis of~\cite{FFL1}. The second is a unimodularly transformed version of the Gelfand--Tsetlin polytope considered in various forms in~\cite{R,KM,M1,MY}, it provides a somewhat less-known but also interesting monomial basis. We generalize these bases to the MCOP setting (note that MCOPs are naturally embedded into $\bR^P$).
\begin{customthm}{C}[cf. Theorem~\ref{basismain}]
For every $O\sqcup C$ and $\la$ there exists a unimodular transformation $\xi$ of $\bR^P$ such that $\{f^{\xi(x)} v_\la\}_{x\in \cO_{O,C}(\la)\cap\bZ^P}$ is a basis in $V_\la$.
\end{customthm}

Finally, we give an addendum where we broaden the scope of our method by applying it in the semi-infinite setting. Semi-infinite Grassmannians are varieties of infinite type introduced in~\cite{FF,FiM} and playing an important role in the theory of Weyl modules, see~\cite{BF,FeM,Kat,DF}. They also parametrize rational curves in Grassmannians, this viewpoint is studied in~\cite{So,SoS} where the term ``quantum Grassmannian'' is used. Less is known about toric degenerations of these varieties, until recently the only known construction was due to~\cite{SoS}. In~\cite{FMP} another toric degeneration was obtained, there the two degenerations were interpreted via poset polytopes of a certain infinite poset $Q$, similarly to the Gelfand--Tsetlin and FFLV degenerations. These poset polytopes are chain-order polytopes: simpler versions of MCOPs depending only on a partition $O\sqcup C=Q$. Here we introduce the notion of semi-infinite pipe dreams which associates a permutation with every finite subset of $Q$ and using this notion prove the following.
\begin{customthm}{D}[cf.\ Theorem~\ref{infmain}]    
For every partition $O\sqcup C=Q$ for which $O$ contains a certain distinguished subset the corresponding chain-order polytope provides a toric degeneration of the semi-infinite Grassmannian.
\end{customthm}
Note that each such degeneration immediately provides a combinatorial character formula for every Weyl module whose highest weight is a multiple of the respective fundamental weight. Indeed, the semi-infinite Grassmannian's coordinate ring is a sum of the duals of such Weyl modules (\cite{BF,FeM,Kat}), hence a toric degeneration of this ring expresses the character as a sum over a polytope's lattice points.

\textbf{Acknowledgments.} The author would like to thank the Weizmann Institute of Science for giving a much-needed refuge and a warm welcome to Russian mathematicians.

\section{Preliminaries and generalities}\label{preliminaries}

\subsection{Degenerations of Pl\"ucker algebras}

We start by recalling some general facts about the Pl\"ucker embedding and its Gr\"obner theory. Consider integers $n>1$ and $1\le d_1<\dots<d_l\le n-1$, denote the tuple $(d_1,\dots,d_l)$ by $\bd$. Let $F_\bd$ be the variety of flags of signature $\bd$ in $\bC^n$. The Pl\"ucker embedding realizes $F_\bd$ as a subvariety in \[\bP_\bd=\bP(\wedge^{d_1}\bC^n)\times\dots\times\bP(\wedge^{d_l}\bC^n).\] The multihomogeneous coordinate ring of $\bP_\bd$ is $S_\bd=\bC[X_{i_1,\dots,i_k}]_{k\in\bd,1\le i_1<\dots<i_k\le n}$, it is equipped with a $\bZ^{n-1}$-grading $\grad$ with $\grad X_{i_1,\dots,i_k}$ the $k$th basis vector $\varepsilon_k$. $F_\bd$ is the zero set of the $\grad$-homogeneous \textit{Pl\"ucker ideal} $I_\bd\subset S_\bd$ and the \textit{Pl\"ucker algebra} $S_\bd/I_\bd$ is the multihomogeneous coordinate ring of $F_\bd$ with respect to the Pl\"ucker embedding. Alternatively, one can say that $F_\bd$ is the variety $\MultiProj{S_\bd/I_\bd}$ with respect to the induced $\bZ^{n-1}$-grading. 

The Pl\"ucker ideal is realized as the kernel of \[\varphi:S_\bd\to T=\bC[z_{i,j}]_{1\le i\le n-1,1\le j\le n}\] with $\varphi(X_{i_1,\dots,i_k})=D_{i_1,\dots,i_k}$ where the latter denotes the minor of the matrix $(z_{i,j})_{i,j}$ spanned by rows $1,\dots,k$ and columns $i_1,\dots,i_k$. We will denote the image $\varphi(S_\bd)$ by $R_\bd$, it is, of course, isomorphic to the Pl\"ucker algebra. The ring $T$ is also equipped with a $\bZ^{n-1}$ grading $\grad$ with $\grad z_{1,j}=\varepsilon_1$ and $\grad z_{i,j}=\varepsilon_i-\varepsilon_{i-1}$ for $i>1$. The map $\varphi$ is $\grad$-homogeneous. Below the induced $\bZ^{n-1}$ gradings on quotients and subspaces of $S_\bd$ and $T$ will also be denoted by $\grad$. The graded component of a space $U$ corresponding to $\la\in\bZ^{n-1}$ will be denoted by $U[\la]$.

Next, for any polynomial ring $\bC[x_a]_{a\in A}$ we understand a \textit{monomial order} $<$ on $\bC[x_a]_{a\in A}$ to be a partial order on the set of monomials with the following two properties.
\begin{itemize}
\item The order is multiplicative: for any two monomials one has $M_1<M_2$ if and only if $M_1x_a<M_2x_a$ for all $a\in A$. 
\item The order is weak, i.e.\ incomparability is an equivalence relation.
\end{itemize}
Note that every total order is weak. In general, weak orders arise as inverse images of total orders. Moreover, every monomial order can be obtained by applying a monomial specialization and then comparing the results lexicographically. We will not be using this general fact, see~\cite[Theorem 1.2]{KNN} for a proof and further context (note that they use the term ``monomial preorder''). 

For a monomial order $<$ and a polynomial $p\in \bC[x_a]_{a\in A}$ one considers the initial part $\initial_< p$ equal to the sum of those monomials occurring in $p$ which are maximal with respect to $<$, taken with the same coefficients as in $p$. For any subspace $U\subset\bC[x_a]_{a\in A}$ one considers its initial subspace $\initial_< U$ spanned by all $\initial_< p$ with $p\in U$. One easily checks that the initial subspace of an ideal is an ideal (the initial ideal) and the initial subspace of a subalgebra is a subalgebra (the initial subalgebra). 

We will also use the notion of sagbi bases where ``sagbi'' stands for ``subalgebra analogue of Gr\"obner bases in ideals''.
\begin{definition}
For a monomial order $<$ on $\bC[x_a]_{a\in A}$ and a subalgebra $U\subset\bC[x_a]_{a\in A}$ a generating set $\{s_b\}_{b\in B}\subset U$ is a \textit{sagbi basis} of $U$ if $\{\initial_< s_b\}_{b\in B}$ generates $\initial_< U$.
\end{definition}

We will chiefly use these Gr\"obner-theoretic notions in two contexts: initial ideals of the Pl\"ucker ideal and initial subalgebras of the Pl\"ucker algebra. A connection between these contexts is given by the below definition and fact.
\begin{definition}
For a monomial order $<$ on $T$ let $\varphi_<:S_\bd\to T$ denote the monomial map given by $\varphi_<(X_{i_1,\dots,i_k})=\initial_< D_{i_1,\dots,i_k}$.  Let $<^\varphi$ denote the monomial order on $S_\bd$ given by $M_1<^\varphi M_2$ if and only if $\varphi_<(M_1)<\varphi_<(M_2)$, i.e.\ the inverse image of $<$ under $\varphi_<$.
\end{definition}

\begin{proposition}\label{idealsubalg}
Consider a monomial order $<$ on $T$ for which the determinants $D_{i_1,\dots,i_k}$ form a sagbi basis of $R_\bd$. Then $\initial_{<^\varphi} I_\bd=\ker\varphi_<$ and $S_\bd/\initial_{<^\varphi} I_\bd$ is isomorphic to $\initial_< R_\bd$.
\end{proposition}
\begin{proof}
Note that the sagbi basis assumption precisely means that $\varphi_<(S_\bd)=\initial_< R_\bd$ so it suffices to prove the first claim. For $p\in S_\bd$ every monomial appearing in $\varphi_<(\initial_{<^\varphi} p)$ appears in $\varphi(p)$ with the same coefficient. Hence, for $p\in I_\bd$ we must have $\varphi_<(\initial_{<^\varphi} p)=0$ and we obtain $\initial_{<^\varphi} I_\bd\subset\ker\varphi_<$. Now, it is well known that passing to an initial subspace preserves the dimensions of graded components. In particular, for any $\la\in\bZ^{n-1}$ we have $\dim(\initial_< R_\bd[\la])=\dim(R_\bd[\la])$. We deduce \[\dim(\ker{\varphi_<}[\la])=\dim(\ker{\varphi}[\la])=\dim(I_\bd[\la])=\dim(\initial_{<^\varphi} I_\bd[\la]).\] Since these dimensions are finite, the above inclusion cannot be strict.
\end{proof}

One sees that the above proposition and its proof generalize straightforwardly to arbitrary monomial maps between polynomial rings whose kernel is homogeneous with finite dimensional components with respect to some grading.

The geometric motivation for considering initial ideals and subalgebras is that they provide flat degenerations. The following fact is essentially classical and holds for any ideal in a polynomial ring, for a proof in the setting of partial monomial orders see~\cite[Theorem 3.2 and Lemma 3.3]{KNN}.

\begin{theorem}\label{flatfamily}
For any monomial order $\ll$ on $S_\bd$ there exists a flat $\bC[t]$-algebra $\tilde R$ such that $\tilde R/\langle t\rangle\simeq S_\bd/\initial_\ll I_\bd$ while for any nonzero $c\in\bC$ we have $\tilde R/\langle t-c\rangle\simeq S_\bd/I\simeq R_\bd$. 
\end{theorem}

In geometric terms this means that we have a flat family over $\mathbb A^1$ for which the fiber over 0 is isomorphic to the zero set of $\initial_\ll I_\bd$ in $\bP_\bd$ while all other fibers are isomorphic to $F_\bd$. This flat family is known as a \textit{Gr\"obner degeneration}. Proposition~\ref{idealsubalg} shows that if we choose a monomial order $<$ on $T$ for which the $D_{i_1,\dots,i_k}$ form a sagbi basis and set $\ll=<^\varphi$, then the fiber over 0 can be identified with $\MultiProj\initial_<R_\bd$, this special case is known as a \textit{sagbi degeneration}. The case of a total $<$ is of particular interest, since $\initial_< R_\bd$ is then generated by a finite set of monomials, i.e.\ it is a toric ring and $\initial_{<^\varphi}I_\bd$ is a toric ideal. This means that the fiber over 0 is a toric variety and we have a toric degeneration of the flag variety.

\subsection{Monomial bases in representations}

Consider the Lie algebra $\fg=\fsl_n(\bC)$, choose a Cartan subalgebra $\fh$. Let $\alpha_1,\dots,\alpha_{n-1}\in\fh^*$ denote the simple roots and denote the positive roots by \[\alpha_{i,j}=\alpha_i+\dots+\alpha_{j-1},~1\le i<j\le n.\] The root vector in $\fg$ corresponding to the negative root $-\alpha_{i,j}$ will be denoted $f_{i,j}$. 

For an integral dominant weight $\lambda$ the irreducible representation with highest weight $\la$ will be denoted by $V_\la$, its highest weight vector by $v_\la$. Let $\om_1,\dots,\om_{n-1}$ denote the fundamental weights, we will identify an integral weight \[\la=a_1\om_1+\dots+a_{n-1}\om_{n-1}\] with the tuple $(a_1,\dots,a_{n-1})\in\bZ^{n-1}$. In particular, if such $\la$ is dominant, it is well known that $\dim V_\la=\dim(R[\la])$ where $R=R_{(1,\dots,n-1)}$ is the Pl\"ucker algebra of the complete flag variety. (Moreover, $R[\la]$ is the contragredient dual of $V_\la$ with respect to the natural $\fg$-action on $R$.) 

We also recall that for $1\le k\le n-1$ the fundamental representation $V_{\om_k}$ is isomorphic to the exterior power $\wedge^k\bC^n$ equipped with the natural $\fg$-action. Specifically, let $e_1,\dots,e_n$ denote the basis vectors in $\bC^n$, then $V_{\om_k}$ is spanned by vectors $e_{i_1,\dots,i_k}=e_{i_1}\wedge\dots\wedge e_{i_k}$. The highest weight vector is $v_{\om_k}=e_{1,\dots,k}$ and $f_{i,j}e_{i_1,\dots,i_k}$ is zero unless $i\in\{i_1,\dots,i_k\}$ and $j\notin\{i_1,\dots,i_k\}$, in which case $f_{i,j}e_{i_1,\dots,i_k}=e_{j_1,\dots,j_k}$ where $(j_1,\dots,j_k)$ is obtained from $(i_1,\dots,i_k)$ by replacing $i$ with $j$.

Let $P$ denote the set of pairs $(i,j)$ with $1\le i\le j\le n$. We use the following shorthand: for a vector $c=(c_{i,j})_{(i,j)\in P}\in\bZ_{\ge 0}^P$ we denote \[z^c=\prod_{(i,j)\in P\backslash\{(n,n)\}}z_{i,j}^{c_{i,j}}\in T\quad\text{and}\quad f^c=\prod_{1\le i<j\le n}f_{i,j}^{c_{i,j}}\] where the second product lies in the universal enveloping algebra $\cU(\fg)$ and the factors are ordered by $i$ increasing from left to right and within a given $i$ by $j$ increasing from left to right (or, equivalently in view of commutation relations, first by $j$ and then by $i$). For instance, when $n=4$ one has $f^c=f_{1,2}^{c_{1,2}}f_{1,3}^{c_{1,3}}f_{1,4}^{c_{1,4}}f_{2,3}^{c_{2,3}}f_{2,4}^{c_{2,4}}f_{3,4}^{c_{3,4}}$. We note that these notations will only be used for $c$ with $c_{n,n}=0$.

\begin{definition}
A monomial order $<$ on $T$ is \textit{triangular} if it is total and has the following two properties.
\begin{itemize}
\item The determinants $D_{i_1,\dots,i_k}$ form a sagbi basis of $R$ with respect to $<$.
\item For every $D_{i_1,\dots,i_k}$ there exists a vector $c^{i_1,\dots,i_k}\in\bZ_{\ge 0}^P$ such that $\initial_< D_{i_1,\dots,i_k}=z^{c^{i_1,\dots,i_k}}$. In other words, the initial parts contain only $z_{i,j}$ with $i\le j$. 
\end{itemize}
\end{definition}

One immediately observes that if $<$ is triangular, then the entire initial subalgebra $\initial_< R$ is spanned by monomials of the form $z^c, c\in\bZ_{\ge 0}^P$, i.e.\ containing only $z_{i,j}$ with $i\le j$. 

\begin{theorem}\label{triangularbasis}
Let $<$ be a triangular monomial order on $T$. Then for any integral dominant weight $\la$ the set of vectors $f^cv_\la$ for which $z^c\in\initial_<R[\la]$ is a basis in $V_\la$.
\end{theorem}
\begin{proof}
We first check the claim for fundamental representations. Indeed, denote \[\initial_< D_{i_1,\dots,i_k}=z^{c^{i_1,\dots,i_k}},\] then, in view of $z^{c^{i_1,\dots,i_k}}$ being a summand in the determinant, we necessarily have \[f^{c^{i_1,\dots,i_k}}e_{1,\dots,k}\in \{e_{i_1,\dots,i_k},-e_{i_1,\dots,i_k},0\}.\] Hence, we only need to check that $f^{c^{i_1,\dots,i_k}}e_{1,\dots,k}\neq 0$ which is ensured by the chosen order of factors in $f^{c^{i_1,\dots,i_k}}$.

Now, consider a general $\la=(a_1,\dots,a_{n-1})$. First of all, note that since $\dim R[\la]=\dim V_\la$, we only need to prove the linear independence of our set. Consider the product \[U=V_{\om_1}^{\otimes a_1}\otimes\dots\otimes V_{\om_{n-1}}^{\otimes a_{n-1}}\] and denote its highest weight vector \[u=v_{\om_1}^{\otimes a_1}\otimes\dots\otimes v_{\om_{n-1}}^{\otimes a_{n-1}}.\] The cyclic submodule $\cU(\fg)u$ is isomorphic to $V_\la$. A $\bZ_{\ge 0}^P$-grading $\gr$ on $U$ is defined as follows: we first define a grading on every $V_{\om_k}$ by $\gr e_{i_1,\dots,i_k}=c^{i_1,\dots,i_k}$ and then extend it multiplicatively to the tensor product. Note that the order $<$ may be viewed as a semigroup order on $\bZ_{\ge 0}^P$ given by $c_1<c_2$ if $z^{c_1}<z^{c_2}$.

Choose any $z^c\in\initial_< R[\la]$. We claim that $f^c u$ has a nonzero $\gr$-homogeneous component in grading $c$ while all of its other nonzero $\gr$-homogeneous components lie in gradings $c'>c$. This will imply the desired linear independence. Indeed, consider the decomposition of $f^cu$ into a linear combination of tensor products of the $e_{i_1,\dots,i_k}$, suppose this decomposition contains the vector \[u'=\bigotimes_{k=1}^{n-1} e_{i_1^{k,1},\dots,i_k^{k,1}}\otimes\dots\otimes e_{i_1^{k,a_k},\dots,i_k^{k,a_k}}\] with a nonzero coefficient. This means that we have vectors $b^{k,r}\in\bZ_{\ge 0}^P$ with $1\le k\le n-1$, $1\le r\le a_k$ such that $f^{b^{k,r}}e_{1,\dots,k}=\pm e_{i_1^{k,r},\dots,i_k^{k,r}}$ and $\sum_{k,r} b^{k,r}=c$. Note that the coordinates $b^{k,r}_{i,i}$ can be chosen arbitrarily as long as they sum up to $c_{i,i}$. Let us show that we may choose the $b^{k,r}$ so that $z^{b^{k,r}}$ is a monomial appearing in the determinant $D_{i_1^{k,r},\dots,i_k^{k,r}}$, this already implies $\gr e_{i_1^{k,r},\dots,i_k^{k,r}}\ge b^{k,r}$ and, consequently, $\gr u'\ge c$. The fact that $f^{b^{k,r}}e_{1,\dots,k}\neq 0$ implies that for every $i\in[1,k]$ we can have $b^{k,r}_{i,j}=1$ for at most one pair $i<j$ while all other such $b^{k,r}_{i,j}=0$. Choosing the $b^{k,r}_{i,i}$ accordingly we may assume that for every $i\in[1,k]$ exactly one coordinate $b^{k,r}_{i,j}=1$, note that this choice satisfies $\sum_{k,r} b^{k,r}=c$.
From $f^{b^{k,r}}e_{1,\dots,k}=\pm e_{i_1^{k,r},\dots,i_k^{k,r}}$ we now deduce that $z^{b^{k,r}}$ is indeed a monomial appearing in $D_{i_1^{k,r},\dots,i_k^{k,r}}$.

It remains to show that $f^c u$ has a nonzero component in grading $c$. The fact that $z^c\in\initial_< R[\la]$ implies that it is a product of the $z^{c^{i_1,\dots,i_k}}$. Specifically, we have tuples $(j_1^{k,r},\dots,j_k^{k,r})$ with $1\le k\le n-1$, $1\le r\le a_k$ such that $\sum_{k,r} c^{j_1^{k,r},\dots,j_k^{k,r}}=c$. This means that the decomposition of $f^c u$ contains the following vector with a nonzero coefficient: \[\bigotimes_{k=1}^{n-1} e_{j_1^{k,1},\dots,j_k^{k,1}}\otimes\dots\otimes e_{j_1^{k,a_k},\dots,j_k^{k,a_k}}.\qedhere\]
\end{proof}

Bases in irreducible representations of this form are known as \textit{monomial bases} because every element is obtained from the highest weight vector by the action of a PBW monomial. Perhaps, the best known basis of this kind is the FFLV basis constructed in~\cite{FFL1}, see also Example~\ref{monbasisex}.

\subsection{Marked chain-order polytopes}

We equip the set $P$ with a partial order $\prec$ such that $(i,j)\preceq(i',j')$ if and only if $i\le i'$ and $j\le j'$. The poset $(P,\prec)$ is sometimes referred to as the \textit{Gelfand--Tsetlin poset}. For $n=4$ the Hasse diagram of this poset looks as follows (with the minimal element on the left and the maximal on the right):
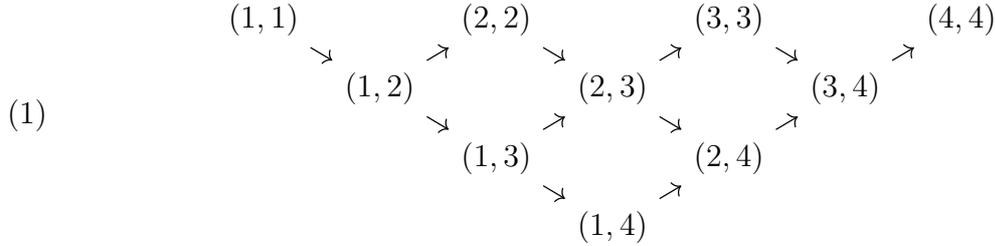
\begin{equation}
\begin{tikzcd}[row sep=tiny,column sep=tiny]\label{hasse}
(1,1)\arrow[rd]&&(2,2)\arrow[rd]&&(3,3)\arrow[rd]&&(4,4)\\
&(1,2)\arrow[rd]\arrow[ru]&&(2,3)\arrow[rd]\arrow[ru]&&(3,4)\arrow[ru]\\
&&(1,3)\arrow[rd]\arrow[ru]&&(2,4)\arrow[ru]\\
&&&(1,4)\arrow[ru]
\end{tikzcd}
\end{equation}

We will now define a family of polytopes associated with this poset. Denote $A=\{(1,1),(2,2),\dots,(n,n)\}$ and for an integral dominant $\fg$-weight $\la=(a_1,\dots,a_{n-1})$ and $1\le i\le n$ denote $\la(i)=a_i+\dots+a_{n-1}$ (so $\la(n)=0$). What follows is a specialization of the definition originally given in~\cite[Definition 1.5]{FF} for arbitrary posets to the case of the poset $(P,\prec)$.
\begin{definition}\label{mcophdef}
Consider an integral dominant $\fg$-weight $\la$ and a partition $O\sqcup C=P\backslash A$. The \textit{marked chain-order polytope} (\textit{MCOP}) $\cO_{O,C}(\la)\subset\bR^P$ consists of points $x$ such that:
\begin{itemize}
\item $x_{i,i}=\la(i)$ for all $1\le i\le n$,
\item $x_{i,j}\ge 0$ for all $(i,j)\in C$,
\item for every chain $(p,q)\prec(i_1,j_1)\prec\dots\prec(i_m,j_m)\prec(r,s)$ with $(p,q),(r,s)\in O\sqcup A$ and all $(i_l,j_l)\in C$ one has \[x_{i_1,j_1}+\dots+x_{i_m,j_m}\le x_{p,q}-x_{r,s}.\]
\end{itemize}
\end{definition}

In the notations of~\cite{FF} this is the polytope $\mathcal{CO}_{O,C}((P,\succ),A,(\la(1),\dots,\la(n)))$. In particular, there it is associated with the dual poset $(P,\succ)$ rather than $(P,\prec)$ but the above form will be more convenient to us. It is worth mentioning that~\cite{FF} imposes a certain restriction on the partition $O\sqcup C$ but the same definition is considered in~\cite[Proposition 1.3]{FFLP} without such restrictions.

One can already note that when $C=\varnothing$ the resulting polytope consists of points with $x_{i,i}=\la(i)$ and $x_{i,j}\ge x_{i',j'}$ whenever $(i,j)\preceq(i',j')$. This is the famous Gelfand--Tsetlin polytope of~\cite{GT}. When $O=\varnothing$ the polytope consists of points with $x_{i,i}=\la(i)$, all $x_{i,j}\ge 0$ and the sum of coordinates over any chain in $C$ starting in $(i,i+1)$ and ending in $(j,j+1)$ being no greater than $\la(i)-\la(j+1)$. This is the FFLV polytope of~\cite{FFL1}. These two polytopes are also known as the corresponding \textit{marked order polytope} and \textit{marked chain polytope}, a terminology due to~\cite{ABS}. For this reason $O$ is known as the ``order part'' and $C$ as the ``chain part''. The other $\cO_{O,C}(\la)$ are seen to interpolate between these two cases.

We will, however, mostly work with an alternative definition which is shown to be equivalent in~\cite[Subsection 3.5]{FM}. The approach here is to explicitly describe the lattice point sets of the fundamental polytopes $\cO_{O,C}(\om_k)$ and then define all other $\cO_{O,C}(\la)$ as Minkowski sums of these. Let $\cJ$ denote the set of order ideals (lower sets) in $(P,\prec)$.
\begin{definition}\label{mcopdef}
Consider a partition $O\sqcup C=P\backslash A$. For $J\in\cJ$ let $x_{O,C}(J)\in\bR^P$ denote the indicator vector $\mathbf 1_{M_{O,C}(J)}$ of the set \[M_{O,C}(J)=(J\cap (O\cup A))\cup\max\nolimits_{\prec}(J)\] ($\max_\prec$ denotes the subset of $\prec$-maximal elements). The MCOP $\cO_{O,C}(\om_k)$ is the convex hull of those $x_{O,C}(J)$ for which $|J\cap A|=k$, i.e.\ $J$ contains $(k,k)$ but not $(k+1,k+1)$. For $\la=(a_1,\dots,a_{n-1})$ the MCOP $\cO_{O,C}(\la)$ is the Minkowski sum \[a_1\cO_{O,C}(\om_1)+\dots+a_{n-1}\cO_{O,C}(\om_{n-1}).\]
\end{definition}

In other words, $M_{O,C}(J)$ contains all elements of $J$ contained in $O\cup A$ and maximal elements of $J$ contained $C$, e.g.\ $M_{P\backslash A,\varnothing}(J)=J$ and $M_{\varnothing,P\backslash A}(J)=(J\cap A)\cup \max_\prec J$. Also, for $n=4$ and $J$ generated by $(1,4)$ and $(2,3)$ the set $M_{O,C}(J)$ consists of the black elements in~\eqref{pipedreamfig} whenever $O$ contains $(1,2)$ but not $(1,3)$.

A key combinatorial property of MCOPs which motivated their introduction is as follows.
\begin{proposition}[{\cite[Corollary 2.5]{FFLP}}]\label{ehrhart}
For a given $\la$ all polytopes $\cO_{O,C}(\la)$ are pairwise \textit{Ehrhart equivalent}, meaning that for every $m\in\bZ_{\ge0}$ the number of integer points in the dilation $m\cO_{O,C}(\la)$ does not depend on $O$ and $C$.
\end{proposition}

\begin{cor}\label{intpoints}
For any $O$, $C$ and $\la$ we have \[|\cO_{O,C}(\la)\cap\bZ^P|=\dim V_\la.\]
\end{cor}
\begin{proof}
By Proposition~\ref{ehrhart} for a given $\la$ it suffices to prove the claim for one chosen partition $O\sqcup C$. However, in the original works~\cite{GT} and~\cite{FFL1} the key property proved for, respectively, the Gelfand--Tsetlin polytope $\cO_{P\backslash A,\varnothing}(\la)$ and the FFLV polytope $\cO_{\varnothing,P\backslash A}(\la)$ is that the polytope's integer points parametrize a basis in $V_\la$. In particular, the number of integer points in either polytope is $\dim V_\la$.
\end{proof}

We will also use the following \textit{Minkowski sum property} of MCOPs.
\begin{proposition}[{\cite[Theorem 2.8]{FFP}}]\label{minkowski}
For integral dominant $\la$ and $\mu$ every integer point in $\cO_{O,C}(\la+\mu)$ decomposes into a sum of two integer points in $\cO_{O,C}(\la)$ and $\cO_{O,C}(\mu)$:
\[\cO_{O,C}(\la+\mu)\cap\bZ^P=(\cO_{O,C}(\la)\cap\bZ^P)+(\cO_{O,C}(\mu)\cap\bZ^P).\]
\end{proposition}

We now proceed to discuss the toric varieties of MCOPs and their multiprojective realizations. We will denote the toric variety of $\cO_{O,C}(\la)$ by $H_{O,C}(\la)$ and call it a \textit{generalized Hibi variety} (see below). For a weight $\la=(a_1,\dots,a_{n-1})$ let its \textit{signature} be the tuple of those $i$ for which $a_i\neq 0$. We have the following simple observation.
\begin{proposition}\label{strcombequiv}
If integral dominant weights $\la$ and $\mu$ have the same signature, then for any $O,C$ the polytopes $\cO_{O,C}(\la)$ and $\cO_{O,C}(\mu)$ are strongly combinatorially equivalent (have the same normal fan). In particular, $H_{O,C}(\la)$ and $H_{O,C}(\mu)$ are isomorphic. 
\end{proposition}
\begin{proof}
This follows directly from Definition~\ref{mcopdef}, since the normal fan of a Minkowski sum is the common refinement of the summands' normal fans.
\end{proof}

For any subset $\cL\subset\cJ$ we will write $\bC[\cL]$ to denote the ring of polynomials in variables $X_J, J\in\cL$. Let $\cJ_k\subset\cJ$ denote the set of $J$ with $|J\cap A|=k$. The ring $\bC[\cJ_k]$ is the homogeneous coordinate ring of the space $\bP(\bC^{\cJ_k})$. Proposition~\ref{minkowski} implies that every MCOP is \textit{normal}, i.e.\ every integer point in the dilation $m\cO_{O,C}(\la)$ is the sum of $m$ integer points in $\cO_{O,C}(\la)$. An implication of Corollary~\ref{intpoints} is that $\cO_{O,C}(\om_k)$ contains no integer points other than the $x_{O,C}(J)$ with $J\in\cJ_k$. That is since $\dim V_{\om_k}={n\choose k}$ and it is also easily seen that $|\cJ_k|={n\choose k}$ and the $x_{O,C}(J)$ are pairwise distinct. By standard properties of toric varieties (see~\cite[\S2.3]{CLS}) these two facts provide a projective embedding of the toric variety. Consider the map $\varphi_{O,C}:\bC[\cJ]\to T$ given by \[\varphi_{O,C}(X_J)=z^{x_{O,C}(J)}.\]
\begin{proposition}
$H_{O,C}(\om_k)$ is realized in $\bP(\bC^{\cJ_k})$ as the zero set of $\ker(\varphi_{O,C}|_{\bC[\cJ_k]})$.
\end{proposition}

Consider an integral dominant weight $\la$ of signature $\bd=(d_1,\dots,d_l)$ and denote $\cJ_\bd=\cJ_{d_1}\cup\dots\cup\cJ_{d_l}$. The ring $\bC[\cJ_\bd]$ is the multihomogeneous coordinate ring of the product \[\bP^{O,C}_\bd=\bP(\bC^{\cJ_{d_1}})\times\dots\times\bP(\bC^{\cJ_{d_l}}).\] By Proposition~\ref{strcombequiv} the variety $H_{O,C}(\la)$ is isomorphic to $H_{O,C}(\om_{d_1}+\dots+\om_{d_l})$. The general multiprojective realization of the toric variety of a Minkowski sum of normal polytopes (see, for instance,~\cite[Lemma 1.8.3]{FM}) provides the following theorem which is a special case of~\cite[Theorem 3.2.3]{FM} for the more general marked relative poset polytopes.
\begin{theorem}\label{hibiproj}
$H_{O,C}(\la)$ is realized in $\bP^{O,C}_\bd$ as the zero set of $I^{O,C}_\bd=\ker(\varphi_{O,C}|_{\bC[\cJ_\bd]})$.
\end{theorem}

One notices that $\bP^{O,C}_\bd$ is isomorphic to the space $\bP_\bd$ containing the flag variety $F_\bd$. In particular one may hope to find a bijection between order ideals and Pl\"ucker variables providing an isomorphism $\psi:\bC[\cJ_\bd]\to S_\bd$ that would map $I^{O,C}_\bd$ to an initial ideal of $I_\bd$. If such an isomorphism is found, we have realized the generalized Hibi variety as a flat degeneration of the flag variety. This was essentially the method of obtaining the flat degenerations given by the Gelfand--Tsetlin polytope in~\cite{GL} and by the FFLV polytope in~\cite{FFFM}. One of our main goals is to generalize this approach to all $H_{O,C}(\la)$.


Note that every $\cJ_\bd$ can be viewed as a distributive lattice with $\cap$ as meet, $\cup$ as join and $\subset$ as the order relation. The ideal $I^{P\backslash A,\varnothing}_\bd\subset\bC[\cJ_\bd]$ is seen to be generated by the binomials $X_{J_1}X_{J_2}-X_{J_1\cup J_2}X_{J_1\cap J_2}$ and is known as the \textit{Hibi ideal} of the lattice $\cJ_\bd$ (after \cite{H}). Hence, the $I^{O,C}_\bd$ are referred to as \textit{generalized Hibi ideals}. Furthermore, the toric ring $\varphi_{P\backslash A,\varnothing}(\bC[\cJ_\bd])$ is known as the \textit{Hibi ring} of the lattice $\cJ_\bd$. We will denote $R_\bd^{O,C}=\varphi_{O,C}(\bC[\cJ_\bd])$ and refer to these rings as \textit{generalized Hibi rings}. The $\Proj$ of the Hibi ring is known as the \textit{Hibi variety} which motivates our terminology for $H_{O,C}(\la)$.

\subsection{Pipe dreams}

There exists a well-known method of associating a permutation to every subset of $P$. Consider the permutation group $\mathcal S_n$ and for $(i,j)\in P$ let $s_{i,j}$ denote the transposition $(i,j)\in\mathcal S_n$. In particular, $s_{i,i}$ is always the identity.
\begin{definition}
For any subset $M\subset P$ let $w_M\in\mathcal S_n$ denote the product of all $s_{i,j}$ with $(i,j)\in M$ ordered by $i$ increasing from left to right and within a given $i$ by $j$ increasing from left to right.
\end{definition}

For instance, for $n=4$ we have (in one-line notation):
\begin{equation}\label{pipedreamex}
w_{\{(1,1),(2,2),(1,2),(2,3),(1,4)\}}=s_{1,1}s_{1,2}s_{1,4}s_{2,2}s_{2,3}=(4,3,1,2).    
\end{equation} 
One may also note that $w_P=w_0$ is the longest permutation. It is obvious that $w_M$ is uniquely determined by $M\backslash A$ but it is convenient for us to consider subsets of $P$ rather than $P\backslash A$.

The permutation $w_M$ can be found with the use of a diagram known as a \textit{pipe dream} (the terminology is due to~\cite{KnM} but similar diagrams were already considered in~\cite{BB} where they were termed \textit{RC-graphs}). In terms of the visualization of $P$ used in~\eqref{hasse} the pipe dream corresponding to $M$ consists of $n$ polygonal curves or \textit{pipes} described as follows. The $i$th pipe enters the element $(i,n)$ from the bottom-right, continues in this direction until it reaches an element of $M\cup A$, after which it turns left and continues going to the bottom-left until it reaches an element of $M$, after which it turns right and again continues to the top-left until it reaches an element of $M\cup A$, etc. The last element passed by the pipe will have the form $(1,j)$. It is straightforwardly checked by induction on $|M|$ that one has $w_M(i)=j$. The pipe dream of the set in~\eqref{pipedreamex} is shown below, here elements of the set are highlighted in black and each pipe is shown in its own colour.
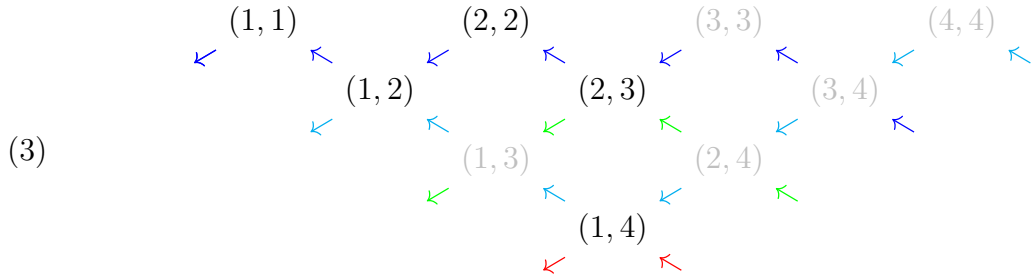
\begin{equation}\label{pipedreamfig}
\begin{tikzcd}[row sep=tiny,column sep=tiny]
&(1,1)\ar[blue]{ld}\ar[blue]{ld}&&(2,2)\ar[blue]{ld}&&\color{lightgray}{(3,3)}\ar[blue]{ld}&&\color{lightgray}{(4,4)}\ar[cyan]{ld}\\
\phantom{(1,1)}&&(1,2)\ar[blue]{lu}\ar[cyan]{ld}&&(2,3)\ar[green]{ld}\ar[blue]{lu}&&\color{lightgray}{(3,4)}\ar[blue]{lu}\ar[cyan]{ld}&&\phantom{(1,1)}\ar[cyan]{lu}\\
&\phantom{(1,1)}&&\color{lightgray}{(1,3)}\ar[green]{ld}\ar[cyan]{lu}&&\color{lightgray}{(2,4)}\ar[green]{lu}\ar[cyan]{ld}&&\phantom{(1,1)}\ar[blue]{lu}\\
&&\phantom{(1,1)}&&(1,4)\arrow[red]{ld}\ar[cyan]{lu}&&\phantom{(1,1)}\ar[green]{lu}\\
&&&\phantom{(1,1)}&&\phantom{(1,1)}\arrow[red]{lu}
\end{tikzcd}
\end{equation}

\begin{remark}
The correspondence between subsets and permutations as well as its visualization are subject to many variations and dualizations, the conventions chosen here differ from those in~\cite{BB} and~\cite{KnM}. For readers familiar with other approaches the following alternative characterization of $w_M$ may be helpful. For $1\le i\le n-1$ let $s_i$ be the elementary transposition $(i,i+1)$ and set $s_0=\id$. Then one may check that $w_M=w'w_0$ where $w'$ is the product of $s_{j-i}$ over all $(i,j)\in P\backslash M$ with the factors ordered by $i$ increasing from left to right and within a given $i$ by $j$ increasing from left to right. Since the pipes cross precisely at elements of $P\backslash (M\cup A)$, one sees that $w_M$ is the product of the simple reflections corresponding to these crossings (in the spirit of~\cite{BB,KnM}) which is then reversed by multiplying by $w_0$.
\end{remark}

\section{Toric degenerations}\label{toric}

In this section we define an isomorphism between $\bC[\cJ_\bd]$ and $S_\bd$ mapping the toric ideal $I^{O,C}_\bd$ to an initial ideal of $I_\bd$. This realizes the generalized Hibi variety as a flat degeneration of the partial flag variety.

Fix a partition $O\sqcup C=P\backslash A$ and a signature $\bd$. We use the shorthand $w_{M_{O,C}(J)}=w^J$. We also use the standard convention $X_{i_1,\dots,i_k}=(-1)^\sigma X_{i_{\sigma(1)},\dots,i_{\sigma(k)}}$ for a permutation $\sigma\in\mathcal S_k$ to define Pl\"ucker variables with subscripts which are not increasing. Similarly, we set $D_{i_1,\dots,i_k}=(-1)^\sigma D_{i_{\sigma(1)},\dots,i_{\sigma(k)}}$.

Consider the homomorphism $\psi:\bC[\cJ_\bd]\to S_\bd$ given by \[\psi(X_J)=X_{w^J(1),\dots,w^J(|J\cap A|)}.\] The key fact proved in this section is as follows.
\begin{theorem}\label{degenmain}
The map $\psi$ is an isomorphism and there exists a total monomial order $\lessdot$ on $T$ for which the $D_{i_1,\dots,i_k}$ form a sagbi basis of $R_\bd$ and $\psi(I_\bd^{O,C})=\initial_{\lessdot^\varphi} I_\bd$. In particular, $\psi$ induces an isomorphism between $R_\bd^{O,C}$ and $\initial_\lessdot R_\bd$.
\end{theorem}

Of course, the isomorphism claim amounts to $J\mapsto\{w^J(1),\dots,w^J(|J\cap A|)\}$ being a bijection between (proper nonempty) ideals in $(P,\prec)$ and (proper nonempty) subsets in $[1,n]$. Now, in view of Theorem~\ref{hibiproj} together with Theorem~\ref{flatfamily} and the subsequent discussion we have the following consequence.
\begin{cor}
For an integral dominant weight $\la$ of signature $\bd$ the toric variety $H_{O,C}(\la)$ is a flat degeneration of the flag variety $F_\bd$.
\end{cor}

\begin{example}\label{psiex}
Let $n=4$ as in~\eqref{pipedreamfig} and suppose $2\in\bd$. Let $J$ be the order ideal generated by $(1,4)$ and $(2,3)$. For any $O$ containing $(1,2)$ and not containing $(1,3)$ the set $M_{O,C}(J)$ and its pipe dream will be as in~\eqref{pipedreamfig}, hence we will have $\psi(X_J)=X_{4,3}$. 
\end{example}

Theorem~\ref{degenmain} will be proved by constructing an isomorphism between $R_\bd^{O,C}$ and an initial subring of $R_\bd$. Before defining the corresponding map we will need the following notions. First, for $(i,j)\in P$ let $\langle (i,j)\rangle\in\cJ$ denote the principal order ideal composed of all $p\le (i,j)$. 

\begin{definition}
Consider $1\le i,j\le n$. If $i\le j$, let $r(i,j)$ denote the integer $w^{\langle(i,j)\rangle}(i)$. If $i>j$, set $r(i,j)=w^P(j)=w_O(j)$.
\end{definition}

The number $r(i,j)$ can be thought of as follows. If $i\le j$, one considers a pipe which starts from $(i,j)$ going to the bottom-left and continues according to the ``zigzag'' procedure used to define pipe dreams, turning at elements of $O\cup A$. The last element passed by this curve will be $(1,r(i,j))$. If $i>j$, one considers the pipe entering $(j,n)$ from the bottom-right and turning at elements of $O\cup A$. 
\begin{proposition}
For every $i$ the numbers $r(i,1),\dots,r(i,n)$ form a permutation of $[1,n]$.
\end{proposition}
\begin{proof}
This is best seen diagrammatically. Let us follow the pipe considered above in the opposite direction: consider a pipe entering $(1,r(i,j))$ from the bottom-left and turning at elements of $O\cup A$. If $j\ge i$, the element $(i,j)$ will be the first element of the form $(i,\cdot)$ passed by this pipe. If $j<i$, then $(j,n)$ will be the last element passed by this pipe. From this characterization it is easily seen that the $r(i,j)$ are pairwise distinct.
\end{proof}

\begin{definition}
Let $T^\pm$ denote the ring of Laurent series $\bC[z_{i,j}^{\pm1}]_{1\le i,j\le n}$. We define an endomorphism $\theta$ of $T^\pm$ as follows. Choose a variable $z_{i,j}$. If $i\ge j$, we set $\theta(z_{i,j})=z_{i,r(i,j)}$. If $i<j$, i.e.\ $(i,j)\in P\backslash A$, we consider the largest $j'<j$ such that $(i,j')\in O\cup A$ and set $\theta(z_{i,j})=z_{i,r(i,j)}/z_{i,r(i,j')}$. 
\end{definition}

\begin{proposition}\label{thetaJ}
For any $J\in\cJ_\bd$ one has \[\theta(z^{x_{O,C}(J)})=z_{1,w^J(1)}\dots z_{|J\cap A|,w^J(|J\cap A|)}.\]
\end{proposition}
\begin{proof}
Evidently, for a given $i_0\le |J\cap A|$ the product of all $\theta(z_{i_0,j})$ with $(i_0,j)\in M_{O,C}(J)$ is equal to $z_{i,r(i_0,j_0)}$ where $j_0$ is maximal among such $j$. However, one has $w^J(i_0)=r(i_0,j_0)$ as seen directly from the pipe dream visualizations for $w^J$ and $r(i_0,j_0)$. 
\end{proof}

\begin{definition}\label{orderdef}
We define a total monomial order $\lessdot$ on $T$ as follows. First, we set $z_{i_1,j_1}\gtrdot z_{i_2,j_2}$ when $i_1<i_2$. Next, for $j_1<j_2$ we set $z_{i,r(i,j_1)}\gtrdot z_{i,r(i,j_2)}$ if and only if $(i,j_1)\in O\cup A$ and there is no $j\in[j_1+1,j_2]$ such that $(i,j)\in O\cup A$. In other words, for some $i$ let $l_1>\dots>l_m=i$ denote all $l$ such that $(i,l)\in O\cup A$, then we have
\begin{multline*}
z_{i,r(i,l_1)}\gtrdot z_{i,r(i,n)}\gtrdot \dots\gtrdot z_{i,r(i,l_1+1)}\gtrdot z_{i,r(i,l_2)}\gtrdot z_{i,r(i,l_1-1)}\gtrdot \dots\gtrdot z_{i,r(i,l_2+1)}\gtrdot \dots\\\gtrdot z_{i,r(i,l_m)}\gtrdot z_{i,r(i,l_{m-1}-1)}\gtrdot \dots\gtrdot z_{i,r(i,i+1)}\gtrdot z_{i,r(i,i-1)}\gtrdot \dots\gtrdot z_{i,r(i,1)}.
\end{multline*}
Finally, we extend $\lessdot$ to a lexicographic order: for monomials $M_1=\prod z_{i,j}^{b_{i,j}}$, $M_2=\prod z_{i,j}^{c_{i,j}}$ let $z_{i,j}$ be the $\lessdot$-maximal variable for which $b_{i,j}\neq c_{i,j}$ and set $M_1\gtrdot M_2$ if $b_{i,j}>c_{i,j}$.
\end{definition}

\begin{example}
When $O=P\backslash A$ we have $r(i,j)=j-i+1$ for $j\le i$ and $r(i,j)=n-j+1$ for $j<i$. In particular, $\lessdot$ first compares the $z_{i,j}$ by $i$ and within a given $i$ we have \[z_{i,n-i+1}\gtrdot z_{i,n-i}\gtrdot\dots\gtrdot z_{i,1}\gtrdot z_{i,n-i+2}\gtrdot z_{i,n-i+3}\gtrdot\dots\gtrdot z_{i,n}.\] When $O=\varnothing$ we have $r(i,j)=j$ for all $j$. Hence, we have \[z_{i,i}\gtrdot z_{i,n}\gtrdot z_{i,n-1}\gtrdot\dots\gtrdot z_{i,i+1}\gtrdot z_{i,i-1}\gtrdot\dots\gtrdot z_{i,1}.\] The initial parts $\initial_\lessdot D_{i_1,\dots,i_k}$ produced by these two orders are described in Example~\ref{degenex}. Also, for $O=\{(1,2),(1,4),(2,3)\}$, i.e.\ the black elements of $P\backslash A$ in~\eqref{pipedreamfig}, below is a table containing the values $r(i,j)$ and the resulting order on the variables:
\begin{center}
\begin{tabular}{c|c|c|c|c}
    \diagbox{$i$}{$j$} &4&3&2&1  \\ \hline
    1& 4& 3& 2& 1\\ \hline
    2& 2& 3& 1& 4\\ \hline
    3& 2& 1& 3& 4\\ \hline
    4& 2& 1& 3& 4
\end{tabular}
\hspace{3cm}
$
\begin{aligned}
&z_{1,4}\gtrdot z_{1,2}\gtrdot z_{1,3}\gtrdot z_{1,1}\gtrdot\\
&z_{2,3}\gtrdot z_{2,2}\gtrdot z_{2,1}\gtrdot z_{2,4}\gtrdot\\
&z_{3,1}\gtrdot z_{3,2}\gtrdot z_{3,3}\gtrdot z_{3,4}\gtrdot\\
&z_{4,2}\gtrdot z_{4,1}\gtrdot z_{4,3}\gtrdot z_{4,4}
\end{aligned}
$
\end{center}
\end{example}

\begin{proposition}\label{initialD}
For $J\in\cJ_\bd$ we have \[\initial_\lessdot D_{w^J(1),\dots,w^J(|J\cap A|)}=z_{1,w^J(1)}\dots z_{|J\cap A|,w^J(|J\cap A|)}.\]
\end{proposition}
\begin{proof}
It suffices to show that $z_{i,w^J(i)}$ is $\lessdot$-maximal among all variables of the form $z_{i,w^J(j)}$ with $i\le j\le |J\cap A|$. Note that $w^J(i)=r(i,l)$ for the maximal $l$ such that $(i,l)\in M_{O,C}(J)$. Now suppose that $z_{i,w^J(j)}\gtrdot z_{i,w^J(i)}$ for some $i< j\le |J\cap A|$. In the pipe dream of $M_{O,C}(J)$ consider the $j$th pipe, it passes through $(1,w^J(j))$. It also passes through some $(i,l')$, consider the minimal such $l'$, then $r(i,l')=w^J(j)$. Note that $(i,l')\in J$. If $l'>l$, then $(i,j')\notin O\cup A$ for all $j'\in[l+1,l']$ and $(i,l)\in J\cap (O\cup A)$ which provides $z_{i,r(i,l')}\lessdot z_{i,r(i,l)}$ contradicting our assumption. Now suppose that $l'<l$. If $(i,l')\notin O\cup A$, we again have $z_{i,r(i,l')}\lessdot z_{i,r(i,l)}$. If $(i,l')\in O\cup A$, then there must exist $(i,l'')\in O\cup A$ with $l''>l'$ which is passed by the $j$th pipe prior to $(i,l')$. We also have $(i,l'')\in J$, therefore, our choice of $l$ implies $l''\le l$, hence $l''\in[l'+1,l]$ and $z_{i,r(i,l')}\lessdot z_{i,r(i,l)}$.
\end{proof}

\begin{proposition}\label{thetasagbi}
The $D_{i_1,\dots,i_k}$ form a sagbi basis of $R_\bd$ for $\lessdot$.
\end{proposition}
\begin{proof}
Propositions~\ref{thetaJ} and~\ref{initialD} provide $\initial_\lessdot D_{w^J(1),\dots,w^J(|J\cap A|)}=\theta(z^{x_{O,C}(J)})$ and we are to show that these elements generate $\initial_\lessdot R_\bd$. Since the $\theta(z^{x_{O,C}(J)})$ generate $\theta(R^{O,C}_\bd)$, we have $\theta(R^{O,C}_\bd)\subset \initial_\lessdot R_\bd$. We prove that $\theta(R^{O,C}_\bd)[\la]=\initial_\lessdot R_\bd[\la]$ for every $\la$. Proposition~\ref{thetaJ} implies that for $J\in\cJ_\bd$ we have $\grad\theta(z^{x_{O,C}(J)})=\om_{|J\cap A|}$. Consequently, if one chooses $a_i$ integer points in every $\cO_{O,C}(\om_i)$, then the sum $x$ of these $a_1+\dots+a_{n-1}$ points satisfies $\grad\theta(z^x)=\la$. However, by Proposition~\ref{minkowski} every $x\in\cO_{O,C}(\la)\cap\bZ^P$ can be expressed as a such sum, hence satisfies $\theta(z^x)\in \theta(R^{O,C}_\bd)[\la]$. Furthermore, the monomials $\theta(z^x)$ with $x\in\cO_{O,C}(\la)\cap\bZ^P$ are pairwise distinct since $\theta$ is seen to be an automorphism from its definition. Corollary~\ref{intpoints} now provides \[\dim\theta(R^{O,C}_\bd)[\la]\ge\dim V_\la=\dim R_\bd[\la]=\dim\initial_\lessdot R_\bd[\la].\qedhere\]
\end{proof}

\begin{proof}[Proof of Theorem~\ref{degenmain}]
For distinct $J_1,J_2\in\cJ_\bd$ the sets $\{w^{J_1}(1),\dots,w^{J_1}(|J_1\cap A|)\}$ and $\{w^{J_2}(1),\dots,w^{J_2}(|J_2\cap A|)\}$ are distinct. Otherwise, Propositions~\ref{thetaJ} and~\ref{initialD} would imply $\theta(z^{x_{O,C}(J_1)})=\pm\theta(z^{x_{O,C}(J_2)})$ which is impossible. This proves the isomorphism claim. Propositions~\ref{thetaJ} and~\ref{initialD} also show that the maps $\theta\circ\varphi_{O,C}$ and $\varphi_\lessdot \circ\psi$ coincide on $\bC[\cJ_\bd]$. Since $\psi$ and $\theta$ are injective, $\psi$ must identify $\ker\varphi_{O,C}=I^{O,C}_\bd$ and $\ker\varphi_\lessdot =\initial_{\lessdot^\varphi}I_\bd$ (the latter by Propositions~\ref{thetasagbi} and~\ref{idealsubalg}). Meanwhile, $\theta$ must identify the maps' images $R^{O,C}_\bd$ and $\initial_\lessdot R_\bd$.
\end{proof}

\begin{example}\label{degenex}
Let $O=P\backslash A$. Then for $1\le i_1<\dots<i_k<n$ one has $\initial_\lessdot D_{i_1,\dots,i_k}=z_{1,i_k}\dots z_{k,i_1}$, i.e.\ this monomial order is ``antidiagonal'' and we obtain the initial subalgebra $\initial_\lessdot R_\bd$ and initial ideal $\initial_{\lessdot^\varphi} I_\bd$ which define the Gelfand--Tsetlin toric degeneration studied in~\cite{Stu,GL,KM}. Now let $O=\varnothing$. Then $\initial_\lessdot D_{i_1,\dots,i_k}=\pm z_{1,\alpha_1}\dots z_{k,\alpha_k}$ where $(\alpha_1,\dots,\alpha_k)$ is the permutation of $(i_1,\dots,i_k)$ that forms a \textit{PBW tuple}: $\alpha_j=j$ if $\alpha_j\le k$ while all $\alpha_j>k$ are ordered decreasingly. The corresponding initial ideal and initial subalgebra define the FFLV toric degeneration and were studied in~\cite{FFFM} (see also~\cite[Section 6]{M2} and~\cite[Subsection 3.2]{FMP}).
\end{example}

\begin{remark}
Results in~\cite{M2} (see~\cite[Corollary 5.4]{M2}) show that the generalized Hibi ideal $I^{O,C}_\bd$ is generated by quadratic binomials \[X_{J_1}X_{J_2}-X_{J_1\cup J_2}X_{J_1*_{O,C}J_2}\] where $J_1,J_2\in\cJ_\bd$ and $*_{O,C}$ is a certain binary operation on $\cJ_\bd$. By applying $\psi$ to these binomials we obtain a set of generators for the toric initial ideal $\psi(I^{O,C}_\bd)$.
\end{remark}

\begin{remark}
It may be curious to compute how many different toric degenerations we obtain from this construction. For instance, for $n=3$ and $\bd=(1,2)$ we have 8 different partitions $(O,C)$, however, the ideal $I_{(1,2)}$ is principal and has only 3 initial toric ideals all of which lie in the same $\mathcal S_3$-orbit (see Section~\ref{monbasis}), in particular, the corresponding toric varieties are pairwise isomorphic. All 3 of these initial ideals have the form $\psi(I_{(1,2)}^{O,C})$ for some $(O,C)$. For $n=4$ and $\bd=(1,2,3)$ the situation is only slightly more interesting: the 64 partitions $(O,C)$ provide 24 distinct initial ideals $\psi(I_\bd^{O,C})$ contained in 2 different $\mathcal S_4$-orbits, one containing $\psi(I_\bd^{P\backslash A,\varnothing})$, the other $\psi(I_\bd^{\varnothing,P\backslash A})$. Thus, we only obtain 2 degenerations up to isomorphism: the Gelfand--Tsetlin and the FFLV toric varieties. The first degenerations which are not isomorphic to either of these appear for $n=5$. For $\bd=(1,2,3)$ one obtains 4 pairwise non-isomorphic toric varieties given by $O=\varnothing$, $O=\{(1,2)\}$, $O=\{(1,3)\}$ and $O=P\backslash A$ (or $\{(1,2),(1,3)\}$). For $\bd=(1,2,3,4)$ computations become rather resource-consuming but the number of isomorphism classes is seen to be between 6 and 8.
\end{remark}

\section{Standard monomials and Young tableaux}

The generalized Hibi ring $R^{O,C}_\bd$ possesses a natural basis parametrized by increasing chains of order ideals. In this section we apply the map $\psi$ to this basis, obtaining a monomial basis in the Pl\"ucker algebra (i.e.\ a \textit{standard monomial theory}) which is parametrized by Young tableaux satisfying a certain standardness condition.

Consider the monomial ideal $I^M_\bd\subset\bC[\cJ_\bd]$ generated by all $X_{J_1}X_{J_2}$ with $J_1\not\subset J_2$ and $J_2\not\subset J_1$. In other words, monomials not in $I^M_\bd$ are those of the form $X_{J_1}\dots X_{J_m}$ with $J_1\subset\dots\subset J_m$. The following is a straightforward observation (see~\cite[Proposition 5.3]{M2}).
\begin{proposition}
There exists a monomial order $\ll$ such that $\initial_\ll I^{O,C}_\bd=I^M_\bd$. Consequently, monomials not lying in $I^M_\bd$ project to a basis in $R^{O,C}_\bd$.
\end{proposition}

In view of Theorem~\ref{degenmain}, by applying $\psi$ we obtain an initial monomial ideal $\psi(I^M_\bd)$ of $I_\bd$ and a set of monomials in $S_\bd$ projecting to a basis in the Pl\"ucker algebra $R_\bd$. We call these \textit{$(O,C)$-standard monomials}, in this section we aim to give a more explicit description of this set in terms of Young tableaux.

\begin{remark}
Geometrically, the ideal $\psi(I^M_\bd)$ defines a flat degeneration of $F_\bd$. However, this degeneration does not, up to isomorphism, depend on $(O,C)$, only its embedding into $\bP_\bd$ does. One easily sees that it is a union of products of projective spaces enumerated by maximal chains in $(\cJ_\bd,\subset)$ and $I^M_\bd$ is, in fact, the Stanley--Reisner ideal of the latter poset. This degeneration appears in~\cite{BL,FL,FM,CFL}.
\end{remark}

The characterization we give relies on the following notion. For $1\le i\le n$ let $\sigma_i$ denote the permutation inverse to $(r(i,1),\dots,r(i,n))$. The value of $\sigma_i(j)$ can be found by considering a pipe entering $(1,j)$ from the bottom-left and turning at elements of $O\cup A$. If this pipe contains an element of the form $(i,l)$, then the first such element passed by the pipe is $(i,\sigma_i(j))$. If the pipe does not contain such an element, then the last element passed by the pipe is $(\sigma_i(j),n)$.

\begin{definition}\label{tupledef}
For $1\le k\le n-1$ we call a tuple of pairwise distinct integers $(i_1,\dots,i_k)$ in $[1,n]$ an \textit{$(O,C)$-tuple} if 
the following hold:
\begin{itemize}
\item $\sigma_j(i_j)\ge j$ for $1\le j\le k$;
\item for any $1\le j<l\le k$ either $\sigma_{j+1}(i_j)=j$ or $\sigma_{j+1}(i_j)>\sigma_l(i_l)$.
\end{itemize}
\end{definition}

\begin{example}
For $O=P\backslash A$ one has $\sigma_j(l)=l+j-1$ for $l\le n+1-j$ and $\sigma_j(l)=n+1-l$ for $l>n+1-j$. The first condition in the definition then means that $i_j\le n+1-j$. For $l=j+1$ the second condition means that either $i_j=n+1-j$ (which implies $i_j>i_l$ via the first condition) or $i_j>i_l$, hence the tuple decreases. It is also seen that both conditions are satisfied for a decreasing tuple, thus $(P\backslash A,\varnothing)$-tuples are precisely the decreasing ones. 

For $O=\varnothing$ one has $\sigma_j(l)=l$ for all $j$, $l$. The second condition implies that all $i_j\neq j$ are ordered decreasingly. Moreover, all $i_j\neq j$ satisfy $i_j>i_k\ge k$, the latter by the first condition. We see that $(i_1,\dots,i_k)$ is a PBW tuple. It is also seen that both conditions are satisfied for a PBW tuple, thus $(\varnothing,P\backslash A)$-tuples are precisely the PBW tuples.
\end{example}

\begin{proposition}\label{tuples}
The set of tuples $(w^J(1),\dots,w^J(|J\cap A|))$ with $J\in\cJ$ and $|J\cap A|\in[1,n-1]$ coincides with the set of all $(O,C)$-tuples.
\end{proposition}
\begin{proof}
Let us show that every $(w^J(1),\dots,w^J(|J\cap A|))=(i_1,\dots,i_k)$ is an $(O,C)$-tuple. Note that $\sigma_j(i_j)$ is equal to the largest $l$ for which $(j,l)\in M_{O,C}(J)$, this immediately provides the first condition. 
Choose $1\le j<l\le |J\cap A|$. If $(j,\sigma_j(i_j))\notin O$, then $(j,\sigma_j(i_j))\in\max_\prec J$. We obtain $\sigma_{j+1}(i_j)=\sigma_j(i_j)>\sigma_l(i_l)$: the equality by $(j,\sigma_j(i_j))\notin O$ and the inequality by $(j,\sigma_j(i_j))\in\max_\prec J$. If $(j,\sigma_j(i_j))\in O$ and there exists $m>\sigma_j(i_j)$ with $(j,m)\in O$, then $\sigma_{j+1}(i_j)=m$ for the minimal such $m$. Since $(j,m)\notin J$, we must have $\sigma_l(i_l)<m$. Finally, if $(j,\sigma_j(i_j))\in O$ and there is no $m>\sigma_j(i_j)$ with $(j,m)\in O$, then $\sigma_{j+1}(i_j)=j$.

Conversely, let $(i_1,\dots,i_k)$ be an $(O,C)$-tuple and let $J$ be the minimal order ideal containing all $(j,\sigma_j(i_j))$ (which lie in $P$ by the first condition). We claim that $w^J(j)=i_j$ for all $j\le k$. Indeed, we are to show that $\sigma_j(i_j)$ is equal to the largest $l$ for which $(j,l)\in M_{O,C}(J)$. This is deduced from the second condition similarly to the above by considering the same three cases for the element $(j,\sigma_j(i_j))$.
\end{proof}

In particular, we see that the set of $(O,C)$-tuples of length $k$ has size $n\choose k$ with each $k$-subset occurring once. This follows from $\psi$ establishing a bijection between the variables in $\bC[\cJ_k]$ and in $S_{(k)}$.

\begin{definition}
Consider a Young tableau $Y$ in English notation with $m$ columns and the $i$th column of height $k_i$. Let $Y_{i,j}$ denote the element in the $j$th cell from the top in the $i$th column. We say that $Y$ is \textit{$(O,C)$-semistandard} if 
\begin{itemize}
\item $(Y_{i,1},\dots,Y_{i,k_i})$ is an $(O,C)$-tuple for every $i\in[1,m]$ and
\item for any $i'\le i$ in $[1,m]$ and $j\in[1,k_i]$ there exists $j'\in[j,k_{i'}]$ for which $\sigma_{j'}(Y_{i',j'})\ge\sigma_j(Y_{i,j})$.
\end{itemize}
\end{definition}

\begin{example}
Consider the case $n=4$ and $\bd=(2,3)$. Let $O=\{(1,2),(1,4),(2,3)\}$, i.e.\ the black elements in~\eqref{pipedreamfig} that lie in $P\backslash A$. We have $\sigma_1=(1,2,3,4)$, $\sigma_2=(2,4,3,1)$ and $\sigma_3=(3,4,2,1)$. The one-column $(O,C)$-semistandard tableaux of heights 2 and 3 are the $(O,C)$-tuples which can be found using either Definition~\ref{tupledef} or Proposition~\ref{tuples}:
\begin{center}
\begin{ytableau}
2\\
1\\
\end{ytableau}
\quad
\begin{ytableau}
3\\
1\\
\end{ytableau}
\quad
\begin{ytableau}
2\\
3\\
\end{ytableau}
\quad
\begin{ytableau}
4\\
1\\
\end{ytableau}
\quad
\begin{ytableau}
4\\
3\\
\end{ytableau}
\quad
\begin{ytableau}
4\\
2\\
\end{ytableau}
\quad
\begin{ytableau}
2\\
3\\
1\\
\end{ytableau}
\quad
\begin{ytableau}
4\\
3\\
1\\
\end{ytableau}
\quad
\begin{ytableau}
4\\
2\\
1
\end{ytableau}
\quad
\begin{ytableau}
4\\
3\\
2
\end{ytableau}
\end{center}
Below on the left we have three examples of tableaux which are $(O,C)$-semistandard and on the right three which are not.
\begin{center}
\begin{ytableau}
4&2&2&3\\
3&3&3&1\\
2&1
\end{ytableau}   
\ 
\begin{ytableau}
4&4&2&2\\
2&3&3&3\\
1
\end{ytableau}   
\ 
\begin{ytableau}
4&2&3\\
2&3&1
\end{ytableau}   
\quad\quad
\begin{ytableau}
4&4&2\\
3&1&3\\
1
\end{ytableau}  
\ 
\begin{ytableau}
2&4\\
3&1
\end{ytableau} 
\  
\begin{ytableau}
4&3&4\\
3&2&2\\
2
\end{ytableau}  
\end{center}
\end{example}

\begin{example}
The $(P\backslash A,\varnothing)$-semistandard tableaux are a dualized version of semistandard Young tableaux: those in which elements decrease strictly in every column from top to bottom and decrease non-strictly in every row from left to right. The $(\varnothing,P\backslash A)$-semistandard tableaux are the PBW-semistandard tableaux introduced in~\cite{Fe}: those for which the elements in every column form a PBW tuple and for any $2\le i\le m$ and $1\le j\le k_i$ there exists $j\le j'\le k_{i-1}$ for which $Y_{i-1,j'}\ge Y_{i,j}$
\end{example}

\begin{theorem}\label{standardmain}
The $(O,C)$-standard monomials are precisely the monomials \[X_{i^1_1,\dots,i^1_{k_1}}\dots X_{i^m_1,\dots,i^m_{k_m}}\in S_\bd\] with $k_1\ge\dots\ge k_m$ for which the Young tableau $Y$ with $Y_{j,l}=i^j_l$, i.e\ with $j$th column $(i^j_1,\dots,i^j_{k_j})$, is $(O,C)$-semistandard. In particular, the set of such monomials projects to a basis in the Pl\"ucker algebra $R_\bd$. 
\end{theorem}
\begin{proof}
We are to show that for $J_1,J_2\in\cJ$  with $|J_1\cap A|\ge|J_2\cap A|$ the two-column tableau with columns $(w^{J_1}(1),\dots,w^{J_1}(|J_1\cap A|))$ and $(w^{J_2}(1),\dots,w^{J_2}(|J_2\cap A|))$ is $(O,C)$-semistandard if and only if $J_2\subset J_1$. The theorem will then follow from Proposition~\ref{tuples} and the preceding discussion. However, every $J\in\cJ$ can be characterized as the minimal order ideal containing all $(j,\sigma_j(w^J(j)))$ with $j\le |J\cap A|$. Hence, $J_2\subset J_1$ if and only if for every $1\le j\le |J_2\cap A|$ there exists $j\le j'\le |J_1\cap A|$ with $(j',\sigma_{j'}(w^{J_1}(j')))\succeq (j,\sigma_j(w^{J_2}(j)))$, i.e.\ $\sigma_{j'}(w^{J_1}(j'))\ge \sigma_j(w^{J_2}(j))$.
\end{proof}

We also show that the structure of the lattice polytope can be recovered from the set of $(O,C)$-semistandard tableaux. Consider the set $\overline P=\{(i,j)|i\in[1,n-1],j\in[1,n]\}$, for $x\in\bZ^{\overline P}$ we have a monomial $z^x\in T^\pm$. 
For a Young tableau $Y$ let $x(Y)\in\bR^{\overline P}$ be the point for which $x(Y)_{i,j}$ is the number of elements equal to $j$ in row $i$ of $Y$.
\begin{proposition}\label{tableauxpolytope}
For an integral dominant weight $\la=(a_1,\dots,a_{n-1})$ the polytope $\cO_{O,C}(\la)$ is unimodularly equivalent to the convex hull of points $x(Y)$ over all $(O,C)$-semistandard tableaux $Y$ of shape $\la$, i.e.\ with exactly $a_i$ columns of height $i$.
\end{proposition}
\begin{proof}
By the definition of $\theta$ we have a unimodular operator $\zeta$ on $\bR^{\overline P}$ such that $\theta(z^x)=z^{\zeta(x)}$ for $x\in\bZ^P\subset\bZ^{\overline P}$ (see also the next section).
Assume $\la$ has signature $\bd$. Monomials $\varphi_\lessdot(M)$ such that $M$ is $(O,C)$-standard and $\grad M=\la$ compose a basis in $\initial_\lessdot R_\bd[\la]$. 
For such an $M=X_{i^1_1,\dots,i^1_{k_1}}\dots X_{i^m_1,\dots,i^m_{k_m}}$ consider the tableau $Y$ with $Y_{j,l}=i^j_l$. Then $Y$ is $(O,C)$-semistandard, has shape $\la$ and $\varphi_\lessdot(M)=z^{x(Y)}$ by Proposition~\ref{initialD}. Therefore, $\initial_\lessdot R_\bd[\la]$ is spanned by $z^{x(Y)}$ with $Y$ an $(O,C)$-semistandard tableau of shape $\la$. We have also seen (proof of Proposition~\ref{thetasagbi}) that this space is spanned by monomials $\theta(z^x)=z^{\zeta(x)}$ with $x\in\cO_{O,C}(\la)\cap \bZ^P$. Hence the convex hull in consideration equals $\zeta(\cO_{O,C}(\la))$.
\end{proof}

\begin{remark}
To conclude this section let us note that the results in~\cite{FM} can be applied to describe the degenerations of $F_\bd$ intermediate between the toric and the monomial one. Consider an ideal $I'$ such that $I'$ is an initial ideal of $I^{C,O}_\bd$ and $I^M_\bd$ is an initial ideal of $I'$. Such ideals are parametrized by faces of the maximal cone corresponding to $I^M_\bd$ in the Gr\"obner fan of $I^{C,O}_\bd$. Note that $\psi(I')$ is an initial ideal of $\psi(I^{C,O}_\bd)$ and of $I_\bd$, hence the zero set of $\psi(I')$ in $\bP_\bd$ is a flat degeneration of $F_\bd$. This zero set is isomorphic to the zero set of $I'$ in $\bP^{O,C}_\bd$ and such zero sets are described by~\cite[ Theorem 3.3.1]{FM}. It is semitoric with each of its component the toric variety of a certain \textit{marked relative poset polytope}. This is a family of poset polytopes generalizing MCOPs each of which is defined by an order $\prec'$ on $P$ weaker than $\prec$ (rather than by a partition $(O,C)$).
\end{remark}

\section{Monomial bases in representations}\label{monbasis}

Choose an integral dominant weight $\la=(a_1,\dots,a_{n-1})$ with signature $\bd$. The goal of this section is to construct a monomial basis in $V_\la$ parametrized by the integer points in (a unimodular transform of) $\cO_{O,C}(\la)$. This will be done by applying Theorem~\ref{triangularbasis}, however, it is easily seen that already in the case $O=P$ the total monomial order $\lessdot$ on $T$ defined in Section~\ref{toric} is not triangular. Thus, our first goal is to obtain a triangular order from it.

The permutation group $\mathcal S_n$ acts on $S_\bd$ by $w(X_{i_1,\dots,i_k})=X_{w(i_1),\dots,w(i_k)}$ and on $T$ by $w(z_{i,j})=z_{i,w(j)}$ for $w\in\mathcal S_n$. It is well known that the ideal $I$ and the subalgebra $R_\bd$ are preserved by these actions. Recall the order $\lessdot$ introduced in Definition~\ref{orderdef}. For $w\in\mathcal S_n$ we consider a monomial order $\lessdot^w$ on $T$ by setting $w(M_1)\lessdot^w w(M_2)$ if and only if $M_1\lessdot M_2$. 
Let $\tau$ denote the permutation $(w^P)^{-1}=(w_O)^{-1}=\sigma_n$.

\begin{proposition}\label{permutedtriangular}
The order $\lessdot^\tau$ is triangular.
\end{proposition}
\begin{proof}
The fact that the determinants form a sagbi basis for $\lessdot^\tau$ is immediate from the analogous property of $\lessdot$. In view of Proposition~\ref{initialD} we are to show that $\tau(w^J(i))\ge i$ for any $J\in\cJ$ and $i\le|J\cap A|$. However, $w^J(i)=r(i,j)$ for some $j\ge i$. By the definitions of $r$ and $\tau$ we have $\tau(r(i,j'))=j'$ for any $j'\in[1,i-1]$, hence $\tau(r(i,j))\in[i,n]$.
\end{proof}

Before applying Theorem~\ref{triangularbasis} let us define the aforementioned unimodular transform of $\cO_{O,C}(\la)$. The fact that for $i\le j$ we have $\tau(r(i,j))\ge i$ implies that $\tau(\theta(z_{i,j}))=z^c$ for some $c\in\bZ^P$ (where $\tau$ is extended to $T^\pm$). In other words, $\tau\circ\theta$ preserves the subring $\bC[z_{i,j}^{\pm1}]_{(i,j)\in P}$. Furthermore, by the definitions of $\tau$ and $\theta$, both automorphisms of $T^\pm$ act on monomials by applying a linear operator to the degrees, i.e.\ we have a linear operator $\xi$ on $\bR^P$ such that $\tau(\theta(z^c))=z^{\xi(c)}$ for $c\in\bZ^P$. For $(i,j)\in P$ let $\varepsilon_{i,j}$ denote the basis vector $\mathbf 1_{\{(i,j)\}}\in\bR^P$ and if $i<j$, choose the largest $j'<j$ for which $(i,j')\in O\cup A$. Then $\xi$ can be written as 
\begin{equation}\label{xiexplicit}
\xi(\varepsilon_{i,j})=
\begin{cases}
\varepsilon_{i,\tau(r(i,j))}-\varepsilon_{i,\tau(r(i,j'))}\text{ if }i<j,\\
\varepsilon_{i,\tau(r(i,i))}\text{ if }i=j.
\end{cases}
\end{equation}
In particular, the matrix of $\xi$ is unitriangular in the basis consisting of $\varepsilon_{i,j}$ ordered by $\sigma_i\tau^{-1}(j)$. Hence, $\Pi_\la^{O,C}=\xi(\cO_{O,C}(\la))$ is a lattice polytope in $\bR^P$  unimodularly equivalent to $\cO_{O,C}(\la)$. For an alternative characterization of the polytope note that \[\Pi_\la^{O,C}=a_1\Pi_{\om_1}^{O,C}+\dots+a_{n-1}\Pi_{\om_{n-1}}^{O,C}\] and Proposition~\ref{thetaJ} implies that $\Pi_{\om_k}^{O,C}$ is the convex hull of 
\begin{equation}\label{piintpoint}
\{\varepsilon_{1,\tau w^J(1)}+\dots+\varepsilon_{k,\tau w^J(k)}\}_{J\in\cJ_k}.
\end{equation}

\begin{theorem}\label{basismain}
The set of vectors $f^cv_\la$ with $c\in \Pi_\la^{O,C}\cap\bZ^P$ is a basis in $V_\la$.
\end{theorem}
\begin{proof}
By Theorem~\ref{triangularbasis} and Proposition~\ref{permutedtriangular} it suffices to show that $\initial_{\lessdot^\tau}R[\la]$ is spanned by the set of $z^c$ with $c\in\Pi_\la^{O,C}\cap\bZ^P$. We have $\initial_{\lessdot^\tau}R_\bd=\tau(\initial_\lessdot R_\bd)$ and via an application of $\tau^{-1}$ we are to show that $\initial_\lessdot R[\la]$ is spanned by the set of $\theta(z^x)$ with $x\in\cO_{O,C}(\la)\cap \bZ^P$ which follows from Propositions~\ref{thetaJ} and~\ref{initialD}.
\end{proof}

\begin{example}\label{monbasisex}
For $O=\varnothing$ all $\tau(r(i,j))=j$ and from~\eqref{xiexplicit} one sees that $\xi(x)_{i,j}=x_{i,j}$ unless $(i,j)\in A$, i.e.\ $\Pi^{\varnothing,P\backslash A}_\la$ has the same projection onto $\bR^{P\backslash A}$ as the FFLV polytope $\cO_{\varnothing,P\backslash A}(\la)$. Hence, in this case the basis provided by Theorem~\ref{basismain} coincides with the FFLV basis $\{f^cv_\la\}_{c\in\cO_{\varnothing,P\backslash A}(\la)\cap \bZ^P}$ constructed in~\cite{FFL1}. Meanwhile, for $O=P\backslash A$ the map $\xi$ is more complicated and $\Pi^{P\backslash A,\varnothing}_\la$ is a transformed version of the Gelfand-Tsetlin polytope appearing in various forms in the papers~\cite{R,KM,M1} (see also~\cite[Section 14.4]{MS}). It consists of points $x\in\bR^P$ such that all $x_{i,j}\ge 0$, for any $i<j$ one has $\sum_{l=j}^n x_{i,l}-\sum_{l=j+1}^n x_{i+1,l}\le a_i$ and all $\sum_{j=i}^n x_{i,j}=\la(i)$ (see~\cite[Section 2]{M1}). The resulting monomial basis is considered in~\cite{R,M1,MY}, it is also a subset of the Chari--Loktev basis~\cite{ChL} in a local Weyl module over $\fsl_n(\bC)[[t]]$.
\end{example}

\begin{remark}
In~\cite{Fu} it is shown that every $\cO_{O,C}(\la)$ can be realized as a Newton--Okounkov body of the complete flag variety $F=F_{(1,\dots,n-1)}$. This statement can be straightforwardly deduced from the above construction. One of the standard ways of defining a Newton--Okounkov body is via a choice of a line bundle $\cL$ on $F$, a global section $t$ of $\cL$ and a valuation $\nu$ on the field $\bC(F)$ (see~\cite{KaKh,Ka}, this is also the language used in~\cite{Fu}). In these terms we consider the equivariant line bundle $\cL_\la$ on $F_\bd$, an arbitrary section $t$ and the \textit{highest term valuation} defined by $\lessdot^\tau$. The latter means that we identify $\bC(F)$ with $\bC(z_{i,j})_{1\le i<j\le n}$ by viewing $z_{i,j}$ as the coordinate on the open Schubert cell corresponding to root $\alpha_{i,j}$ and set $z^{\nu(f)}=\initial_{\lessdot^\tau} f$ for any polynomial $f$. Then for $m\in\bZ_{\ge0}$ the convex hull of \[\left\{\frac{\nu(s/t^{\otimes m})}{m}\left|s\in H^0(F,\cL_\la^{\otimes m})\right.\right\}\] is a translate of $\Pi^{O,C}_\la$ independent of $m$, i.e.\ the respective Newton-Okounkov body is identified with $\cO_{O,C}(\la)$ by a suitable change of coordinates.
\end{remark}

\begin{remark}
A basis in $V_\la$ parametrized by the lattice points of $\cO_{O,C}(\la)$ is also constructed in~\cite[Theorem 6.8]{Fu}. This basis is essentially different from the one given by Theorem~\ref{basismain} as seen already for $n=3$, $\la=(1,1)$ and $C=\varnothing$. It would be interesting to study the transition matrix between the two bases.
\end{remark}

\section{Addendum: semi-infinite Grassmannians}

In this addendum we show how the above techniques can be applied in the semi-infinite setting to obtain a family of toric degenerations of the semi-infinite Grassmannian.

Choose integers $n>k\ge 1$ and consider the rings \[S_\infty=\bC[X_{i_1,\dots,i_k}^{(l)}]_{1\le i_1<\dots<i_k\le n,\, l\ge 0}\] and \[T_\infty=\bC[z_{i,j}^{(l)}]_{i\in[1,k],\, j\in[1,n],\, l\ge 0}.\] We write $z_{i,j}(t)$ to denote the power series $\sum_{l\ge 0}z_{i,j}^{(l)}t^l$ and let $D_{i_1,\dots,i_k}(t)$ denote the minor of the matrix $(z_{i,j}(t))_{i\in[1,k],j\in[1,n]}$ spanned by columns $i_1<\dots<i_k$. Let $D_{i_1,\dots,i_k}^{(l)}$ denote the coefficient of $t^l$ in the power series $D_{i_1,\dots,i_k}(t)$ and $\varphi_\infty:S_\infty\to T_\infty$ be given by \[\varphi_\infty\left(X_{i_1,\dots,i_k}^{(l)}\right)=D_{i_1,\dots,i_k}^{(l)}.\] The kernel of $\varphi_\infty$ is the \textit{semi-infinite Pl\"ucker ideal} $I_\infty$ and the image of $\varphi_\infty$ is the \textit{semi-infinite Pl\"ucker algebra} $R_\infty$ (the subalgebra generated by all $D_{i_1,\dots,i_k}^{(l)}$). 
The variety of infinite type $\Proj R_\infty$ is known as the \textit{semi-infinite Grassmannian} (due to~\cite{FeFr}) or the \textit{quantum Grassmannian} (due to~\cite{So}). Similarly to the finite case we have the following.

\begin{definition}
For a monomial order $<$ on $T_\infty$ let $\varphi_<:S_\infty\to T_\infty$ denote the map given by $\varphi_<(X_{i_1,\dots,i_k}^{(l)})=\initial_< D_{i_1,\dots,i_k}^{(l)}$. Let $<^\varphi$ denote the monomial order on $S_\infty$ given by $M_1<^\varphi M_2$ if and only if $\varphi_<(M_1)<\varphi_<(M_2)$, i.e.\ the inverse image of $<$ under $\varphi_<$.
\end{definition}

\begin{proposition}\label{infidealsubalg}
Consider a monomial order $<$ on $T_\infty$ for which the $D_{i_1,\dots,i_k}^{(l)}$ form a sagbi basis of $R_\infty$. Then $S_\infty/\initial_{<^\varphi} I_\infty$ is isomorphic to $\initial_< R_\infty$.
\end{proposition}
\begin{proof}
The proof repeats that of Proposition~\ref{idealsubalg} with the difference that, in order to have finite dimensional graded components, we consider a $\bZ\oplus\bZ$ grading given by $\grad_\infty X_{i_1,\dots,i_k}^{(l)}=k\oplus l$ and $\grad_\infty z_{i,j}^{(l)}=1\oplus l$.
\end{proof}

Initial ideals of $I_\infty$ and initial subalgebras of $R_\infty$ provide flat degenerations of the semi-infinite Grassannian, at least in the following simple setting.

\begin{definition}
A total monomial order $<$ on $T_\infty$ is a \textit{reverse lexicographic order} if for monomials $M_1=\prod (z_{i,j}^{(l)})^{b_{i,j}^{(l)}}$ and $M_2=\prod (z_{i,j}^{(l)})^{c_{i,j}^{(l)}}$ we have $M_1<M_2$ if and only if for the $<$-minimal $z_{i,j}^{(l)}$ with $b_{i,j}^{(l)}\neq c_{i,j}^{(l)}$ one has $b_{i,j}^{(l)}>c_{i,j}^{(l)}$.
\end{definition}

\begin{proposition}[{see~\cite[Section 2]{FMP}}]\label{infflatfamily}
Consider a reverse lexicographic order $<$ on $T_\infty$ for which the $D_{i_1,\dots,i_k}^{(l)}$ form a sagbi basis of $R_\infty$. There exists a flat family over $\bC$ with all fibers outside of 0 isomorphic to the semi-infinite Grassmannian $\Proj R_\infty$ and the fiber over 0 isomorphic to $\Proj(\initial_<R_\infty)=\Proj(S_\infty/\initial_{<^\varphi}I_\infty)$.
\end{proposition}

We will be constructing toric degenerations of this form and proceed to define the corresponding combinatorial objects.

\begin{definition}\label{precdef}
Let $(Q,\prec)$ be the poset consisting of elements $(i,j)$ with $i\ge 1$, $j\ge k+1$ and $j-i\in[0,n-1]$. We set $(i_1,j_1)\preceq\ (i_2,j_2)$ if at least one of the following holds:
\begin{enumerate}[label=(\roman*)]
\item $i_1\le i_2$ and $j_1\le j_2$,
\item $i_2-i_1\ge k$ or
\item $j_2-j_1\ge n-k$.
\end{enumerate}
We write $\prec$ for the corresponding strict relation.
\end{definition}
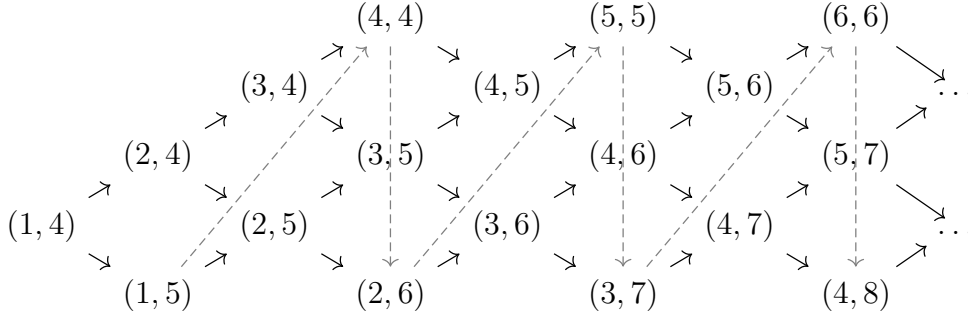
\[
\begin{tikzcd}[row sep=tiny,column sep=tiny]
&&&(4,4)\arrow[dr]\arrow[dddd, dashed, gray] && (5,5)\arrow[dr] \arrow[dddd, dashed, gray]&&(6,6)\arrow[dr]\arrow[dddd, dashed, gray]\\
&&(3,4)\arrow[dr]\arrow[ur]&& (4,5)\arrow[dr]\arrow[ur] && (5,6)\arrow[dr]\arrow[ur] && \dots\\
&(2,4)\arrow[dr]\arrow[ur]&&  (3,5)\arrow[ur]\arrow[dr] && (4,6)\arrow[dr]\arrow[ur] && (5,7)\arrow[dr]\arrow[ur]\\
(1,4)\arrow[ur]\arrow[dr]&&(2,5)\arrow[dr]\arrow[ur] &&    (3,6)\arrow[dr]\arrow[ur] && (4,7)\arrow[dr]\arrow[ur]  && \dots\\
&(1,5)\arrow[ur]\arrow[uuuurr, dashed, gray]&&(2,6)\arrow[ur]\arrow[uuuurr, dashed, gray] && (3,7)\arrow[uuuurr, dashed, gray]\arrow[ur] && (4,8)\arrow[ur]\\
\end{tikzcd}
\]

The definition of $(Q,\prec)$ is due to~\cite[Section 5]{FMP}. Above one sees part of the Hasse diagram of $(Q,\prec)$ in the case $n=5$ and $k=3$ (with the longer arrows dashed for visibility). In general, the covering relations in $Q$ are of four types: 
\begin{enumerate}
\item $(i,j)\prec (i+1,j)$ for $i<j$, 
\item $(i,j)\prec (i,j+1)$ for $j-i<n-1$,
\item $(i,i+n-1)\prec (i+k,i+k)$ for $i\ge 1$ and 
\item $(i+k,i+k)\prec (i+1,i+n)$ for $i\ge 1$.
\end{enumerate}
To understand the nature of the relations of types (3) and (4) consider the set $H$ of pairs $(i+mk,j-m(n-k))$ with $(i,j)\in Q$ and $m\in\bZ$ equipped with a partial order such that $(i_1,j_1)\prec'(i_2,j_2)$ if and only if $i_1\le j_1$ and $i_2\le j_2$. The poset $(H,\prec')$ is acted upon by the group generated by the translation $(i,j)\mapsto (i+k,j-n+k)$ and $(Q,\prec)$ is seen to be the corresponding quotient poset. In particular, for any $(i,j)\in H$ there exists a unique $m\in\bZ$ such that $(i+mk,j-m(n-k))\in Q$. For $(i,j)\in H$ we will write $\langle i,j\rangle$ to denote the corresponding $(i+mk,j-m(n-k))\in Q$. In this notation covering relation (3) can be rewritten as $(i,j)\prec\al i,j+1\ar$ where $j=i+n-1$ and (4) can be rewritten as $(i,j)\prec\al i+1,j\ar$ where $j=i+n$. Hence all covering relations have form $(i,j)\prec\al i+1,j\ar$ or $(i,j)\prec\al i,j+1\ar$ for some $i,j\in\bZ$. We proceed to introduce an analog of pipe dreams for this setting.

\begin{definition}
For $a\in\bZ$ we understand $a\mod k$ to be an integer in $[1,k]$ and $a\mod(n-k)$ to be an integer in $[k+1,n]$. For $i,j\in\bZ$ let $\tilde s_{i,j}$ denote the transposition $(i\mod k,j\mod(n-k))\in\mathcal S_n$. For $M\subset Q$ let $w_M\in\mathcal S_n$ be the product $\tilde s_{i,j}$ over all $(i,j)\in M$ ordered so that $\tilde s_{i_1,j_1}$ is on the left of $\tilde s_{i_2,j_2}$ whenever $(i_1,j_1)\prec(i_2,j_2)$. (Note that $\tilde s_{i_1,j_1}$ and $\tilde s_{i_2,j_2}$ commute unless $i_1-i_2\in k\bZ$ or $j_1-j_2\in(n-k)\bZ$, hence they commute when $(i_1,j_1)$ and $(i_2,j_2)$ are incomparable.)
\end{definition}

While the finite poset $(P,\prec)$ is naturally realized as a set of points in the plane, the poset $(Q,\prec)$ is more naturally realized as a set of points in the plane modulo a translation, i.e.\ a set of points in a cylinder. This makes the diagrammatic intuition here more complicated and instead of defining pipe dreams in terms of polygonal curves we will use the more formal language of paths in the Hasse diagram (note that the polygonal curves above can be viewed as such paths). 
\begin{definition}\label{infpipedreamdef}
We call two edges in the Hasse diagram of $(Q,\prec)$ \textit{parallel} if either both correspond to covering relations of the form $(i,j)\prec\al i+1,j\ar$ or both correspond to covering relations of the form $(i,j)\prec\al i,j+1\ar$. A pipe for a subset $M\subset Q$ is a sequence $(i_1,j_1),\dots,(i_m,j_m)$ of elements of $Q$ with $m>1$ and the following properties.
\begin{enumerate}[label=(\alph*)]
\item $(i_l,j_l)$ covers $(i_{l+1},j_{l+1})$ for $l<m$, so the sequence is a path in the Hasse diagram (with arrows reversed).
\item For $l<m-1$ the edge connecting $(i_l,j_l)$ and $(i_{l+1},j_{l+1})$ is parallel to the edge connecting $(i_{l+1},j_{l+1})$ and $(i_{l+2},j_{l+2})$ if and only if $(i_{l+1},j_{l+1})\notin M$, i.e.\ the ``direction'' changes at elements of $M$.
\item The sequence is maximal in the sense that one may not choose $(i_{m+1},j_{m+1})\in Q$ so that the above is satisfied.
\end{enumerate}
\end{definition}

It is evident that a pipe is determined by its first two elements. Also note that $(i_m,j_m)$ covers no more than one element, hence either $i_m=1$ and $j_m\in[k+1,n]$ or $j_m=k+1$ and $i_m\in[1,k]$.

\begin{definition}
For a pipe $(i_1,j_1),\dots,(i_m,j_m)$ its \textit{value} is an integer in $[1,n]$ defined as follows. If $i_m>1$, then the value is $i_m$. If $j_m>k+1$, then the value is $j_m$. If $(i_m,j_m)=(1,k+1)$, then we have two possibilities. The value is 1 when $(1,k+1)\notin M$ and $(i_{m-1},j_{m-1})=(1,k+2)$ or when $(1,k+1)\in M$ and $(i_{m-1},j_{m-1})=(2,k+1)$. Otherwise, the value is $k+1$. We denote the value by $N_M((i_1,j_1),(i_2,j_2))$.
\end{definition}
One way of thinking about the value is that it is determined by the last element of the pipe together with the ``direction'' in which the pipe would leave this element if it were to continue beyond it. Specifically, add a vertex $(i_{m+1},j_{m+1})$ to the Hasse diagram where $i_{m+1}$ and $j_{m+1}$ are just formal symbols rather than numbers. Let $(i_{m+1},j_{m+1})$ be adjacent to $(i_m,j_m)$ and formally assign the edge between them to one of the two parallelity classes in such a way that property (b) in the definition of a pipe is satisfied for $l=m-1$. If the edge is parallel to edges connecting $(i,j)$ with $(i+1,j)$, then the pipe's value is $j_m$. If the edge is parallel to edges connecting $(i,j)$ with $(i,j+1)$, then the pipe's value is $i_m$.

\begin{lemma}\label{infpipedream}
For finite $M\subset Q$ consider $(C,D)\in Q$ such that $(C,D)\not\preceq(i,j)$ for all $(i,j)\in M$. If $\al C,D-1\ar\in Q$, then
\begin{equation}\label{valuecase1}
N_M((C,D),\al C,D-1\ar)=w_M(C\mod k)    
\end{equation}
and if $\al C-1,D\ar\in Q$, then \[N_M((C,D),\al C-1,D\ar)=w_M(D\mod (n-k)).\]
\end{lemma}
\begin{proof}
This is checked by induction on $|M|$. Let $|M|=0$. To verify the first equality, note that the pipe starting with $(C,D),\al C,D-1\ar$ consists of all $(C',D')\preceq (C,D)$ with $C'-C\in k\bZ$. Its value is $C\mod k$. We also trivially have $w_M(C\mod k)=C\mod k$. The second equality is verified similarly with both sides equal to $D\mod (n-k)$

Now suppose $|M|>0$. If the pipe starting with $(C,D),\al C,D-1\ar$ contains no elements of $M$, then both sides of the first equality are again equal to $C\mod k$ because $i\mod k\neq C\mod k$ for any $(i,j)\in M$. Otherwise, denote the pipe's elements by $(i_1,j_1),\dots,(i_m,j_m)$ and choose the smallest $r$ such that $(i_r,j_r)\in M$. Note that $(i_r,j_r)$ is the $\prec$-maximal element in $M$ with $i_r-C\in k\bZ$. Let $M'$ consist of $(i,j)\in M$ for which $(i,j)\not\succeq(i_r,j_r)$. One sees that either $r=m$ and $i_r=C\mod k=1$ so that both sides of~\eqref{valuecase1} equal $j_r$ or \[N_M((C,D),\al C,D-1\ar)=N_{M'}((i_r,j_r),\al i_r-1,j_r\ar).\] By the induction hypothesis the latter is equal to $w_{M'}(j_r\mod (n-k))$ and we are to show that applying the product of $\tilde s_{i,j}$ with $(i,j)\in M\backslash M'$ to $C\mod k=i_r\mod k$ provides $j_r\mod (n-k)$. However, the above characterization of $(i_r,j_r)$ shows that $\tilde s_{i_r,j_r}$ is the only transposition in the above product which acts nontrivially on $C\mod k$. The induction step for the second equality is similar.
\end{proof}

\begin{example}\label{infpipedreamex}
Consider $n=5$, $k=3$ and $M=\{(1,4),(2,5),(4,4),(3,6),(4,5)\}$. Let us write the first equality in Lemma~\ref{infpipedream} for $(C,D)$ equal to $(4,7)$, $(5,5)$ and $(3,7)$ and the second equality for $(5,6)$ and $(5,5)$. The edges of the corresponding pipes are shown below, each pipe in its own colour, while the elements of $M$ are highlighted in black. The proposition shows that the values of these 5 pipes taken in the above order compose the permutation $w_M$. Hence, $w_M=(2,5,4,3,1)$ which is also equal to $\tilde s_{1,4}\tilde s_{2,5}\tilde s_{4,4}\tilde s_{3,6}\tilde s_{4,5}=s_{1,4}s_{2,5}s_{1,4}s_{3,4}s_{1,5}$. For an arbitrary $M$, we can construct $n$ pipes in a similar way to find $w_M$. The collection of these pipes may be viewed as a ``pipe dream'' of $M$ but, of course, there are infinitely many suitable choices for the pipes' first elements.
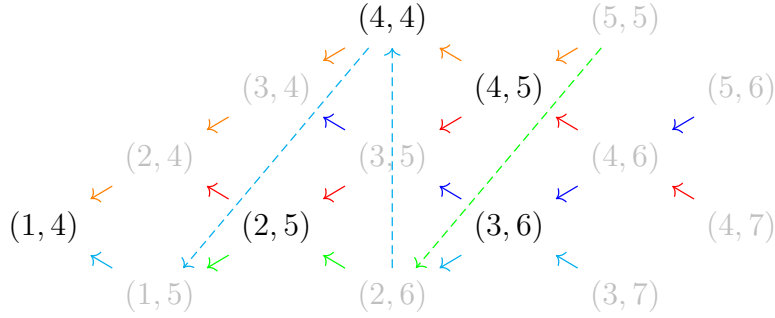
\[
\begin{tikzcd}[row sep=tiny,column sep=tiny]
&&&(4,4)\arrow[dl,orange]\arrow[ddddll, dashed,cyan] && \color{lightgray}(5,5)\arrow[dl,orange] \arrow[ddddll, dashed,green]\\
&&\color{lightgray}(3,4)\arrow[dl,orange]&& (4,5)\arrow[dl,red]\arrow[ul,orange] && \color{lightgray}(5,6)\arrow[dl,blue]\\
&\color{lightgray}(2,4)\arrow[dl,orange]&&  \color{lightgray}(3,5)\arrow[ul,blue]\arrow[dl,red] && \color{lightgray}(4,6)\arrow[dl,blue]\arrow[ul,red]\\
(1,4)&&(2,5)\arrow[dl,green]\arrow[ul,red] &&    (3,6)\arrow[dl,cyan]\arrow[ul,blue] && \color{lightgray}(4,7)\arrow[ul,red]\\
&\color{lightgray}(1,5)\arrow[ul,cyan]&&\color{lightgray}(2,6)\arrow[ul,green]\arrow[uuuu, dashed,cyan] && \color{lightgray}(3,7)\arrow[ul,cyan]\\
\end{tikzcd}
\]
\end{example}

We now define the final ingredient of our theorem: \textit{chain-order polytopes} associated with $(Q,\prec)$. They are termed ``interpolating polytopes'' in~\cite{M2,FMP} but since their relationship to MCOPs is the same as that of order polytopes to marked order polytopes or of chain polytopes to marked chain polytopes, we now consider the former term more appropriate.
\begin{definition}
Let $\cJ_\infty$ denote the set of all finite order ideals in $(Q,\prec)$. For a partition $O\sqcup C$ of $Q$ and $J\in \cJ_\infty$ denote \[M_{O,C}(J)=(J\cap O)\cup\max\nolimits_\prec(J).\] The corresponding chain-order polytope $\cO_{O,C}\subset\bR^Q$ is the convex hull of indicator vectors $\mathbf 1_{M_{O,C}(J)}$ for all $J\in\cJ_\infty$.
\end{definition}

Such a polytope contains no integer points other than the $\mathbf 1_{M_{O,C}(J)}$ and is also normal (see~\cite[Section 3]{FMP}). In particular, its toric variety (the corresponding generalized Hibi variety) is $\Proj R_\infty^{O,C}$ which is defined as follows.
\begin{definition}
Denote $\bC[\cJ_\infty]=\bC[X_J]_{J\in\cJ_\infty}$ and $\bC[Q,s]=\bC[s][z_p]_{p\in Q}$. Consider the map $\varphi_\infty^{O,C}:\bC[\cJ_\infty]\to\bC[Q,s]$ given by \[\varphi_\infty^{O,C}(X_J)=s\prod_{p\in M_{O,C}(J)} z_p.\] The kernel of $\varphi_\infty^{O,C}$ is the generalized Hibi ideal $I_\infty^{O,C}$ and the image of $\varphi_\infty^{O,C}$ is the generalized Hibi ring $R_\infty^{O,C}$.
\end{definition}

Now fix a partition $Q=O\sqcup C$ such that $O$ contains all elements of the form $(i,i)$. For $\sigma\in\mathcal S_k$ denote $X_{i_{\sigma(1)},\dots,i_{\sigma(k)}}^{(l)}=(-1)^\sigma X_{i_1,\dots,i_k}^{(l)}$. For $J\in\cJ_\infty$ denote $w^J=w_{M_{O,C}(J)}$. Define a map $\psi_\infty:\bC[J_\infty]\to S_\infty$ by \[\psi_\infty(X_J)=X_{w^J(1),\dots,w^J(k)}^{(d(J))}\] where $d(J)$ is the number of elements of the form $(i,i)$ contained in $J$. The following is our main result concerning semi-infinite Grassmannians.
\begin{theorem}\label{infmain}
$\psi_\infty$ is an isomorphism and there exists a reverse lexicographic order $\lessdot$ on $T$ for which the $D_{i_1,\dots,i_k}^{(l)}$ form a sagbi basis of $R_\infty$ and $\psi_\infty(I_\infty^{O,C})=\initial_{\lessdot^\varphi} I_\infty$. In particular, $\psi_\infty$ induces an isomorphism between $R_\infty^{O,C}$ and $\initial_\lessdot R_\infty$.
\end{theorem}
By Proposition~\ref{infflatfamily} we have
\begin{cor}
The generalized Hibi variety $\Proj R^{O,C}_\infty$ is a flat degeneration of the semi-infinite Grassmannian $\Proj R_\infty$.
\end{cor}

\begin{example}\label{infpsiex}
As in Example~\ref{infpipedreamex} consider $n=5$ and $k=3$. Let $J\in\cJ_\infty$ be the order ideal generated by $(4,5)$ and $(3,6)$. The example shows that for any $O$ containing $(1,4),(2,5)$ and all $(i,i)$ and not containing $(1,5),(2,4),(3,4),(2,6),(3,5)$ we will have $\psi_\infty(X_J)=X_{2,5,4}^{(1)}$.
\end{example}

The proof will follow the same outline as that of Theorem~\ref{degenmain}.
\begin{definition}
For $(i,j)\in Q$ set $r(i,j)=N_O((i,j),\al i-1,j\ar)$ if the latter is well-defined. If not, we have $i=1$ and set $r(i,j)=j$.
\end{definition}

\begin{definition}
Consider $(i,j)\in Q$, suppose $i=qk+(i\mod k)$. Set $\theta'(z_{(i,j)})=z_{i\mod k,r(i,j)}^{(q)}$ defining a map $\bC[z_p]_{p\in Q}\to T_\infty$. Now define a map \[\theta_\infty:\bC[Q,s]\to \bC\big[{z_{i,j}^{(l)}}^{\pm1}\big]_{i\in[1,k],\, j\in[1,n],\, l\ge 0}\] as follows. For $(i,j)$ as above consider the $\prec$-maximal $(i',j')\in O$ with $(i',j')\prec (i,j)$ and $i'-i\in k\bZ$. If such $(i',j')$ exists, set $\theta_\infty(z_{(i,j)})=\theta'(z_{(i,j)}/z_{(i',j')})$. If not, set $\theta_\infty(z_{(i,j)})=\theta'(z_{(i,j)})/z_{i\mod k,i\mod k}^{(0)}$. Finally, set $\theta_\infty(s)=z_{1,1}^{(0)}\dots z_{k,k}^{(0)}$.
\end{definition}

\begin{example}
For $O$ as in Example~\ref{infpsiex} one sees that $r(4,6)=3$ and $r(4,5)=2$, hence $\theta_\infty(z_{(4,6)})=\theta'(z_{(4,6)}/z_{(4,5)})=z_{1,3}^{(1)}/z_{1,2}^{(1)}$. We also have $r(4,4)=1$ and $r(1,4)=r(3,6)=4$, hence $\theta_\infty(z_{(4,4)})=\theta'(z_{(4,4)}/z_{(1,4)})=z_{1,1}^{(1)}/z_{1,4}^{(0)}$ and $\theta_\infty(z_{(3,6)})=\theta'(z_{(3,6)})/z_{3,3}^{(0)}=z_{3,4}^{(0)}/z_{3,3}^{(0)}$.
\end{example}

\begin{proposition}\label{infthetaJ}
Consider $J\in\cJ_\infty$, let $d(J)\mod k=l$, suppose $d(J)=qk+l$. Then
\[\theta_\infty\left(s\prod_{p\in M_{O,C}(J)}z_p\right)=z_{l+1,w^J(l+1)}^{(q)}\dots z_{k,w^J(k)}^{(q)}z_{1,w^J(1)}^{(q+1)}\dots z_{l,w^J(l)}^{(q+1)}.\]
\end{proposition}
\begin{proof}
Consider $a\in[1,k]$ and let $\Pi_a$ denote the product of $\theta_\infty(z_{(i,j)})$ over all $(i,j)\in M_{O,C}(J)$ with $i-a\in k\bZ$. If this product is nonempty, it is equal to $\theta'(z_{(i_a,j_a)})/z_{a,a}^{(0)}$ where $(i_a,j_a)$ is $\prec$-maximal among all $(i,j)$ in the product. Furthermore, in this case we have $w^J(a)=r(i_a,j_a)$ by Lemma~\ref{infpipedream}.

First, suppose $a\in[l+1,k]$. If the product $\Pi_a$ is nonempty, we must have $i_a=qk+a$. That is since $(i,i)\in M_{O,C}(J)$ if and only if $i\le d(J)+k$ so that $((q+1)k+a,(q+1)k+a)\notin M_{O,C}(J)$ and either $(qk+a,qk+a)\in M_{O,C}(J)$ or $q=0$. We see that $\theta'(z_{(i_a,j_a)})=z_{a,w^J(a)}^{(q)}$. Now, if the product is empty, we have $q=0$ and $w^J(a)=a$. Hence, in both cases $z_{a,a}^{(0)}\Pi_a=z_{a,w^J(a)}^{(q)}$.

Now, if $a\in[1,l]$, then the product cannot be empty because $((q+1)k+a,(q+1)k+a)\in M_{O,C}(J)$. In this case we have $i_a=(q+1)k+a$ and $z_{a,a}^{(0)}\Pi_a=\theta'(z_{(i_a,j_a)})=z_{a,w^J(a)}^{(q+1)}$. We see that $\theta_\infty(s)\prod_{a\in[1,k]}\Pi_a$ is equal to both the left- and right-hand sides in the statement.
\end{proof}

Next, for $1\le i\le j\le k$ (so that $(i,j)\notin Q$) let us set $r(i,j)=j$.
\begin{proposition}
For any $i\ge 1$ the values $r(i,i),\dots,r(i,i+n-1)$ form a permutation of $1,\dots,n$.
\end{proposition}
\begin{proof}
For $i=1$ we evidently have the identity permutation. If $i>1$, choose $(i,j_1),(i,j_2)\in Q$ and note that if $r(i,j_1)=r(i,j_2)$, then the two pipes starting with $(i,j_1),(i-1,j_1)$ and $(i,j_2),(i-1,j_2)$ must end with the same two elements. However, by Definition~\ref{infpipedreamdef} this implies that one pipe contains the other which is impossible. It remains to show that for $i\in[2,k]$ and $j\in[k+1,i+n-1]$ we have $r(i,j)\notin[i,k]$ but in this case $(a,k+1)\not\prec(i-1,j)$ for any $a\in[i,k]$, hence the pipe starting with $(i,j),(i-1,j)$ cannot have value $a$.
\end{proof}

\begin{definition}
Define a reverse lexicographic order $\lessdot$ on $T_\infty$ as follows. First, set $z_{i_1,j_1}^{(l_1)}\lessdot z_{i_2,j_2}^{(l_2)}$ whenever $l_1<l_2$ or $l_1=l_2$ and $i_1<i_2$. Now consider $l\ge 0$ and $i\in[1,k]$, note that $r(kl+i,kl+i),\dots,r(kl+i,kl+i+n-1)$ is a permutation of $1,\dots,n$. For $0\le j_1<j_2\le n-1$ set $z_{i,r(kl+i,kl+i+j_1)}^{(l)}\lessdot z_{i,r(kl+i,kl+i+j_2)}^{(l)}$ unless $(kl+i,kl+i+j_1)\in O\cup\{(1,1),\dots,(k,k)\}$ and there is no $j\in[j_1+1,j_2]$ for which $(kl+i,kl+i+j)\in O$.
\end{definition}

\begin{example}
Let $n=6$, $k=3$. If $(5,5),(5,8)\in O$ and $(5,6),(5,7),(5,9),(5,10)\in C$, then \[z_{2,r(5,6)}^{(l)}\lessdot z_{2,r(5,7)}^{(l)}\lessdot z_{2,r(5,5)}^{(l)}\lessdot z_{2,r(5,9)}^{(l)}\lessdot z_{2,r(5,10)}^{(l)}\lessdot z_{2,r(5,8)}^{(l)}.\] Also, if $(2,6)\in O$ and $(2,4),(2,5),(2,7)\in C$, then \[z_{2,r(2,3)}^{(0)}=z_{2,3}^{(0)}\lessdot z_{2,r(2,4)}^{(0)}\lessdot z_{2,r(2,5)}^{(0)}\lessdot z_{2,r(2,2)}^{(0)}=z_{2,2}^{(0)}\lessdot z_{2,r(2,7)}^{(0)}\lessdot z_{2,r(2,6)}^{(0)}.\]
\end{example}

\begin{proposition}\label{infinitialD}
Consider $J\in\cJ_\infty$, let $d(J)\mod k=l$, suppose $d(J)=qk+l$. Then
\[\initial_\lessdot D_{w^J(1),\dots,w^J(k)}^{(d(J))}=\pm z_{l+1,w^J(l+1)}^{(q)}\dots z_{k,w^J(k)}^{(q)}z_{1,w^J(1)}^{(q+1)}\dots z_{l,w^J(l)}^{(q+1)}.\]
\end{proposition}
\begin{proof}
Since the $z_{i,j}^{(l)}$ are ordered first by $l$ and then by $i$, we immediately see that the initial term must have the form \[\pm z_{l+1,\bullet}^{(q)}\dots z_{k,\bullet}^{(q)}z_{1,\bullet}^{(q+1)}\dots z_{l,\bullet}^{(q+1)}.\] Consider $a\in[l+1,k]$, let us show that $z_{a,w^J(a)}^{(q)}>z_{a,w^J(b)}^{(q)}$ for any $b\in[1,n]\backslash[l+1,a]$. Let $(i_a,j_a)$ and $(i_b,j_b)$ be as in the proof of Proposition~\ref{infthetaJ}. Recall that $w^J(a)=r(i_a,j_a)$ and $w^J(b)=r(i_b,j_b)$. We also have $i_a=qk+a$. Since $i_b>i_a$, we may consider the smallest $j'_b\ge i_a$ such that the pipe starting with $(i_b,j_b),\al i_b-1,j_b\ar$ contains $(i_a,j'_b)$ (the last element of the form $(i_a,j)$ passed by the pipe). Then $w^J(b)=r(i_a,j'_b)$, also note that $(i_a,j'_b)\in J$. Consider the largest $j_0$ with $(i_a,j_0)\in J$. By our choice of $(i_a,j_a)$, either $(i_a,j_a)\in O$ and $(i_a,j)\notin O$ for all $j\in[j_a+1,j_0]$ or $j_a=j_0$. In the former case the variable $z_{a,r(i_a,j_a)}^{(q)}$ is $\lessdot$-maximal among all $z_{a,r(i_a,j)}^{(q)}$ with $(i_a,j)\in J$, hence \[z_{a,w^J(a)}^{(q)}=z_{a,r(i_a,j_a)}^{(q)}>z_{a,r(i_a,j'_b)}^{(q)}=z_{a,w^J(b)}^{(q)}.\] Suppose $j_a=j_0$. If $(i_a,j'_b)\notin O$, then $j'_b=j_b<j_a$ and $z_{a,r(i_a,j_a)}^{(q)}>z_{a,r(i_a,j'_b)}^{(q)}$. If $(i_a,j'_b)\in O$, then, by our choice of $j'_b$, there exists at least one $j'\in[j'_b+1,j_a]$ with $(i_a,j')\in O$: we may take $(i_a,j')$ to be the first element of the form $(i_a,j)$ in the pipe starting with $(i_b,j_b),\al i_b-1,j_b\ar$. This again provides $z_{a,r(i_a,j_a)}^{(q)}>z_{a,r(i_a,j'_b)}^{(q)}$. The case $a\in[1,l]$ and $b\in[a+1,l]$ is similar.
\end{proof}

\begin{proposition}
The $D_{i_1,\dots,i_k}^{(l)}$ form a sagbi basis of $R_\infty$ for $\lessdot$.
\end{proposition}
\begin{proof}
First let us show that $\theta_\infty$ is injective. Since the map is monomial, it suffices to show that distinct monomials have distinct images. Consider distinct $M_1=s^{b_0}\prod z_{(i,j)}^{b_{(i,j)}}$ and $M_2=s^{c_0}\prod z_{(i,j)}^{c_{(i,j)}}$, let us show that $\theta_\infty(M_1)\neq\theta_\infty(M_2)$. We may assume that the GCD of $M_1$ and $M_2$ is 1. We may also assume that $M_1/M_2$ is not a power of $s$ since this case is trivial. Among all $z_{(i,j)}$ with $b_{(i,j)}+c_{(i,j)}>0$ consider the variables with the largest $i$ and among these the variable with the largest $j$, denote it by $z_{(i',j')}$. Assume that $b_{(i',j')}>0$ while $c_{(i',j')}=0$. One sees that the variable $\theta'(z_{(i',j')})$ occurs in $\theta_\infty(M_1)$ but not in $\theta_\infty(M_2)$. 

Now, we are to show that $\theta_\infty(R^{O,C}_\infty)=\initial_\lessdot R_\infty$. By Propositions~\ref{infthetaJ} and~\ref{infinitialD} we have $\theta_\infty(R^{O,C}_\infty)\subset \initial_\lessdot R_\infty$. We claim that $\theta(R^{O,C}_\infty)$ and $R_\infty$ have the same graded dimensions with respect to $\grad_\infty$. By~\cite[Corollary 2.5]{FMP} the ring $R^{O,C}_\infty$ has a basis consisting of products $\pi=\varphi_\infty^{O,C}(X_{J_1}\dots X_{J_m})$ with $J_1\subset\dots\subset J_m$. By Proposition~\ref{infthetaJ} we have $\grad_\infty\theta_\infty(\pi)=km\oplus(d(J_1)+\dots+d(J_m))$. By the injectivity of $\theta_\infty$ this implies that the graded dimension of $\theta_\infty(R^{O,C}_\infty)$ does not depend on $(O,C)$ and it suffices to prove the equality for one partition. However, the case $O=\{(i,i)\}_{i\ge k+1}$ is due to~\cite{FMP}, namely the proof of~\cite[Lemma 5.9]{FMP} precisely shows that $\theta_\infty(R^{O,C}_\infty)=\initial_\lessdot R_\infty$ in this case.
\end{proof}

\begin{proof}[Proof of Theorem~\ref{infmain}]
Consider distinct $J_1,J_2\in\cJ_\infty$ with $d(J_1)=d(J_2)$. The sets $\{w^{J_1}(1),\dots,w^{J_1}(k)\}$ and $\{w^{J_2}(1),\dots,w^{J_2}(k)\}$ must differ, otherwise Propositions~\ref{infthetaJ} and~\ref{infinitialD} would imply $\theta_\infty(s\prod_{p\in M_{O,C}(J_1)}z_p)=\pm\theta_\infty(s\prod_{p\in M_{O,C}(J_2)}z_p)$. One also sees that for $m\ge 0$ there are $n\choose k$ order ideals $J$ with $d(J)=m$. 
The isomorphism claim follows. 

Propositions~\ref{infthetaJ} and~\ref{infinitialD} also show that the maps $\theta_\infty\circ\varphi_\infty^{O,C}$ and $\varphi_\lessdot \circ\psi_\infty$ coincide on $\bC[\cJ_\infty]$. Since $\psi_\infty$ and $\theta_\infty$ are injective, $\psi_\infty$ must identify the kernels $I_\infty^{O,C}$ and $\initial_{\lessdot^\varphi}I_\infty$ of $\varphi_\infty^{O,C}$ and $\varphi_\lessdot$ while $\theta_\infty$ must identify their images $R_\infty^{O,C}$ and $\initial_\lessdot R_\infty$.
\end{proof}


\begin{thebibliography}{}


\bibitem[AB]{AB}
V.~Alexeev, M.~Brion, \textit{Toric degenerations of spherical varieties}, Selecta
Mathematica \textbf{10} (2005), 453--478, 2005.

\bibitem[A]{A}
D.~Anderson, \textit{Okounkov bodies and toric degenerations}, Mathematische Annalen
\textbf{356} (2013), 1183--1202.

\bibitem[ABS]{ABS} 
F. Ardila, T. Bliem, D. Salazar, {\it Gelfand--Tsetlin polytopes and Feigin--Fourier--Littelmann--Vinberg polytopes as marked poset polytopes},
Journal of Combinatorial Theory, Series A {\bf 118} (2011), no. 8, 2454--2462.



\bibitem[BCKS]{BCKS}
V.~Batyrev, I.~Ciocan-Fontanine, B.~Kim, D.~van~Straten,
\textit{Mirror symmetry and toric degenerations of partial flag manifolds}, Acta Mathematica 184 (2000), no. 1, 1--39.

\bibitem[BB]{BB}
N. Bergeron, S. C. Billey, \textit{RC-graphs and Schubert polynomials}, Experimental Mathematics \textbf{2} (1993), 257--269

\bibitem[BL]{BL}
J. Brown, V. Lakshmibai, \emph{Singular loci of Grassmann-Hibi toric varieties}, The Michigan Mathematical Journal \textbf{59} (2010), no. 2, 243--267. 

\bibitem[BLMM]{BLMM}
L. Bossinger, S. Lamboglia, K. Mincheva, F. Mohammadi, \emph{Computing Toric Degenerations of Flag Varieties}. In: G. Smith, B. Sturmfels (eds), \emph{Combinatorial Algebraic Geometry}, Fields Institute Communications, Vol. 80, Springer-Verlag, New York, 2017


\bibitem[BF]{BF}
A. Braverman, M. Finkelberg, {\it Weyl modules and $q$-Whittaker functions}, Math. Ann., {\bf 359} (2014), 45--59.

\bibitem[Ca]{Ca}
P. Caldero, \emph{Toric degenerations of Schubert varieties}, Transformation Groups \textbf{7} (2002), 51--60.

\bibitem[ChL]{ChL} 
V. Chari, S. Loktev, \emph{Weyl, Demazure and fusion modules for the current algebra of $\fsl_{r+1}$}, Advances in Mathematics, \textbf{207} (2006), 928--960.


\bibitem[CFL]{CFL}
R. Chirivi', X. Fang, P. Littelmann, \emph{Seshadri stratifications and standard monomial theory},  Inventiones mathematicae \textbf{234} (2023), 489--572.



\bibitem[Ch]{Ch}
R. Chiriv\`i, \emph{LS algebras and application to Schubert varieties}, Transformation Groups \textbf{5} (2000), 245--264

\bibitem[CM]{CM}
O. Clarke, F. Mohammadi, \textit{Toric degenerations of Grassmannians from matching fields}, Algebraic Combinatorics \textbf{2} (2019), 1109--1124.





\bibitem[CLS]{CLS}
D. Cox, J. Little, H. Schenck, \emph{Toric Varieties}, Graduate Studies in Mathematics, Vol. 124, American Mathematical Society, Providence, 2011.




\bibitem[DF]{DF}
I.~Dumanski, E.~Feigin, {\it Reduced arc schemes for Veronese embeddings and global Demazure modules}, \url{https://arxiv.org/abs/1912.07988}.

\bibitem[FFFM]{FFFM}
X. Fang, E. Feigin, G. Fourier, I. Makhlin, {\it Weighted PBW degenerations and tropical flag varieties}, Communications in Contemporary Mathematics {\bf 21} (2019), no. 1, 1850016.

\bibitem[FF]{FF}
X. Fang, G. Fourier, \textit{Marked chain-order polytopes}, European Journal of Combinatorics \textbf{58} (2016), 267--282.

\bibitem[FaFL1]{FaFL1}
X.~Fang, G.~Fourier, P.~Littelmann, \emph{On toric degenerations of flag varieties}, Representation Theory - Current Trends and Perspectives, EMS Series of Congress Reports, 187--232, 2016.

\bibitem[FaFL2]{FaFL2}
X. Fang, G. Fourier, P. Littelmann, \emph{Essential bases and toric degenerations arising from birational sequences}, Advances in Mathematics {\bf 312} (2017), 107--149.

\bibitem[FFLP]{FFLP}
X. Fang, G. Fourier, J.-P. Litza, C. Pegel, \emph{A Continuous Family of Marked Poset Polytopes}, SIAM Journal on Discrete Mathematics \textbf{34} (2020), no. 1, 611--639.

\bibitem[FFP]{FFP}
X. Fang, G. Fourier, C. Pegel, \emph{The Minkowski Property and Reflexivity of Marked Poset Polytopes}, The Electronic Journal of Combinatorics, \textbf{27} (2020).



\bibitem[FL]{FL}
X. Fang, P. Littelmann, \emph{From standard monomial theory to semi-toric degenerations via Newton--Okounkov bodies}, Transactions of the Moscow Mathematical Society \textbf{78} (2017), no. 2, 331--356.


\bibitem[FeFr]{FeFr}
B. Feigin, E. Frenkel, {\it Affine Kac-Moody algebras and semi-infinite flag manifolds}, Comm. Math. Phys. \textbf{128} (1990), 161--189.

\bibitem[Fe]{Fe}
E.~Feigin, \emph{${\mathbb G}_a^M$ degeneration of flag varieties}, Selecta Mathematica, New Series {\bf 18} (2012), no. 3, 513--537

\bibitem[FFL1]{FFL1} 
E.~Feigin, G.~Fourier, P.~Littelmann,
\emph{PBW filtration and bases for irreducible modules in type ${A}_n$}, Transformation Groups {\bf 16} (2011), 71--89.

\bibitem[FFL2]{FFL2}
E. Feigin, G. Fourier, P. Littelmann, \emph{Favourable modules: filtrations, polytopes, Newton-Okounkov bodies and flat degenerations}, Transformation Groups {\bf 22} (2017), 321--352.




\bibitem[FeM]{FeM}
E.~Feigin, I.~Makedonskyi, {\it Semi-infinite Plücker Relations and Weyl Modules}, Int. Math. Res. Not. IMRN {\bf 14} (2020). 4357--4394.

\bibitem[FM]{FM}
E. Feigin, I. Makhlin, \textit{Relative poset polytopes and semitoric degenerations}, Selecta Mathematica New Series \textbf{30} (2024), 48.

\bibitem[FMP]{FMP}
E. Feigin, I. Makhlin, A. Popkovich, \textit{Beyond the Sottile–Sturmfels Degeneration of a Semi-Infinite Grassmannian}, International Mathematics Research Notices, \url{https://doi.org/10.1093/imrn/rnac116}.

\bibitem[FiM]{FiM}
M. Finkelberg; I. Mirković, {\it Semi-infinite flags. I. Case of global curve $\mathbb{P}^1$}, Differential topology, infinite-dimensional Lie algebras, and applications, Ser. 2 {\bf 194} (1999), 81--112.


\bibitem[FK]{FK}
S. Fomin, A. N. Kirillov, \textit{The Yang-Baxter equation, symmetric functions, and Schubert polynomials}, Discrete Mathematics \textbf{153} (1996), 123--143.

\bibitem[Fu]{Fu}
N.~Fujita, \textit{Newton-Okounkov polytopes of flag varieties and marked chain-order polytopes}, \url{https://arxiv.org/abs/2104.09929}.


\bibitem[GT]{GT}
I. Gelfand, M. Tsetlin, {\it Finite dimensional representations of the group of unimodular matrices}, Doklady Akademii Nauk USSR {\bf 71} (1950), no. 5, 825--828.

\bibitem[GL]{GL}
N. Gonciulea, V. Lakshmibai, \emph{Degenerations of flag and Schubert varieties to toric varieties}, Transformation Groups {\bf 1} (1996), 215--248.

\bibitem[GHKK]{GHKK}
M. Gross, P. Hacking, S. Keel, M. Kontsevich, \textit{Canonical bases for cluster algebras}, Journal of the American Mathematical Society \textbf{31} (2018), 497--608.



\bibitem[H]{H}
T. Hibi, \textit{Distributive lattices, affine semigroup rings and algebras with straightening laws}. In
Commutative Algebra and Combinatorics, 93--109, Advanced Studies in Pure Mathematics,
Vol. 11, Amsterdam, 1987.







\bibitem[Kat]{Kat}
S. Kato, \textit{Demazure character formula for semi-infinite flag varieties}, Mathematische Annalen \textbf{371} (2018), no. 3--4, 1769--1801.

\bibitem[Ka]{Ka}
K. Kaveh, \textit{Crystal bases and Newton--Okounkov bodies}, Duke Mathematical Journal \textbf{164} (2015), 2461--2506.

\bibitem[KaKh]{KaKh}
K. Kaveh, A. G. Khovanskii, \textit{Newton--Okounkov bodies, semigroups of integral points, graded algebras
and intersection theory}, Annals of Mathematics \textbf{176} (2012), 925--978.



\bibitem[KNN]{KNN}
G. Kemper, Ngo V. T., Nguyen T. V. A., \textit{Toward a theory of monomial preorders}, Mathematics of Computation \textbf{87} (2018), 2513--2537.

\bibitem[KST]{KST}
V. Kiritchenko, E. Smirnov, V. Timorin, \textit{Schubert calculus and Gelfand-Zetlin polytopes}, Russian Mathematical Surveys \textbf{67} (2012), 685--719.

\bibitem[KnM]{KnM}
A. Knutson, E.Miller, Gr\"obner geometry of Schubert polynomials, Annals of Mathematics \textbf{161} (2005), 1245--1318.

\bibitem[K]{K}
M. Kogan, \textit{Schubert geometry of flag varieties and Gelfand-Cetlin theory}, Ph.D. thesis, Massachusetts Institute of Technology, 2000.

\bibitem[KM]{KM}
M. Kogan, E. Miller, \emph{Toric degeneration of Schubert varieties and Gelfand–Tsetlin polytopes}, Advances in Mathematics {\bf 193} (2005), no. 1, 1--17.




\bibitem[M1]{M1}
I. Makhlin, \emph{Gelfand--Tsetlin degenerations of representations and flag varieties}, Transformation Groups \textbf{27} (2022), 563--596.

\bibitem[M2]{M2}
I. Makhlin, \emph{Gr\"obner fans of Hibi ideals, generalized Hibi ideals and flag varieties}, Journal of Combinatorial Theory, Series A \textbf{185} (2022), 105541.

\bibitem[MS]{MS}
E. Miller, B. Sturmfels, {\it Combinatorial Commutative Algebra}, Graduate Texts in Mathematics, Vol. 227, Springer-Verlag, New York, 2005.



\bibitem[MY]{MY}
A. Molev, O. Yakimova, \emph{Monomial bases and branching rules}, Transformation Groups \textbf{26} (2021), 995--1024.












\bibitem[R]{R}
M. Reineke, \emph{On the coloured graph structure of Lusztig’s canonical basis}, Math. Ann. {\bf 307} (1997), 705–723.



\bibitem[So]{So}
F. Sottile, \textit{Real rational curves in Grassmannians}, J. Amer. Math. Soc. \textbf{13} (2000), 333--341.

\bibitem[SoS]{SoS}
F. Sottile, B. Sturmfels, {\it A sagbi basis for the quantum Grassmannian}, J. Pure Appl. Algebra , {\bf 158} Issues 2–3 (2001), 347-366.

\bibitem[St]{St}
R.~P.~Stanley, \emph{Two poset polytopes}, Discrete \& Computational Geometry \textbf{1} (1986), 9--23.

\bibitem[Stu]{Stu}
B. Sturmfels, \textit{Algorithms in invariant theory}, Texts Monogr. Symbol. Comput., Springer (1993).





\end{thebibliography}
\end{document}